\documentclass[a4paper,11pt]{article}
\usepackage{amsmath}
\usepackage{amssymb}
\usepackage{theorem}
\usepackage{euscript}
\usepackage{pstricks}
\usepackage{authblk}
 \usepackage{relsize}

\usepackage{tikz}
\usepackage{ifthen}
\usepackage{hyperref}
\hypersetup{
	colorlinks,
	citecolor=black,
	filecolor=black,
	linkcolor=black,
	urlcolor=blue
}
\hypersetup{linktocpage}

\usepackage{color}

\DeclareMathOperator*{\esssup}{ess\,sup}

\newtheorem{theorem}{Theorem}[section]
\newtheorem{lemma}{Lemma}[section]
\newtheorem{corollary}{Corollary}[section]
\numberwithin{equation}{section}
\def\II{{\mathbb I}}

\def\ZZ{{\mathbb Z}}
\def\NN{{\mathbb N}}
\def\RR{{\mathbb R}}

\def\IId{{\mathbb I}^d}

\def\NNd{{\mathbb N}^d}
\def\RRd{{\mathbb R}^d}

\def\NNn{{\mathbb N}_{-1}}
\def\NNdn{{\mathbb N}^d_{-1}}

\def\Ha{H^\alpha_\infty(\IId)}

\def\Uas{\mathring{U}^{\alpha}_\infty}

\def\supp{\operatorname{supp}}

\def\VCdim{\operatorname{VCdim }}

\theoremstyle{definition}
\newtheorem{definition}[theorem]{Definition}

\theoremstyle{remark}
\newtheorem{remark}[theorem]{Remark}

\newcommand{\cO}{\mathcal{O}}

\newcommand{\bI}{\mathbf{I}}

\newcommand{\dd}{\mathrm{d}}

\newcommand{\Chi}{\raise .3ex
	\hbox{\large $\chi$}}

\newcommand{\R}{\mathbb{R}}

\newcommand{\N}{\mathbb{N}}

\usepackage[section]{algorithm}
\usepackage{algorithmicx}
\usepackage{algpseudocode}
\algrenewcommand\algorithmicrequire{\makebox[46pt][l]{\textrm{required:}}}
\algrenewcommand\algorithmicensure{\makebox[46pt][l]{\textrm{output:}}}
\algrenewcommand\algorithmicfunction{\textrm{function}}
\algrenewcommand\algorithmicwhile{\textrm{while}}
\algrenewcommand\algorithmicdo{}
\algrenewcommand\algorithmicend{\textrm{end}}
\algrenewcommand\algorithmicforall{\textrm{for all}}
\algrenewcommand\algorithmicfor{\textrm{for}}
\algrenewcommand\algorithmicrepeat{\textrm{repeat}}
\algrenewcommand\algorithmicuntil{\textrm{until}}
\algrenewcommand\algorithmicif{\textrm{if}}
\algrenewcommand\algorithmicthen{\textrm{then}}
\algrenewcommand\algorithmicelse{\textrm{else}}

\newcommand{\bb}{{\boldsymbol{b}}}

\newcommand{\bee}{{\boldsymbol{e}}}

\newcommand{\bk}{{\boldsymbol{k}}}
\newcommand{\bh}{{\boldsymbol{h}}}

\newcommand{\bs}{{\boldsymbol{s}}}

\newcommand{\bx}{{\boldsymbol{x}}}

\newcommand{\by}{{\boldsymbol{y}}}

\newcommand{\bz}{{\boldsymbol{z}}}

\newcommand{\bW}{{\boldsymbol{W}}}
\newcommand{\bN}{{\boldsymbol{N}}}

\newcommand{\bxi}{{\boldsymbol{\xi}}}
\newcommand{\bnu}{{\boldsymbol{\nu}}}

\newcommand{\bbeta}{{\boldsymbol{\beta}}}
\newcommand{\balpha}{{\boldsymbol{\alpha}}}

\newcommand{\boned}{\boldsymbol{1}}

\usepackage{todonotes}

\newcommand{\be}{\begin{equation}}
\newcommand{\ee}{\end{equation}}
\newcommand{\beq}{\begin{eqnarray}}
\newcommand{\beqq}{\begin{eqnarray*}}
\newcommand{\eeq}{\end{eqnarray}}
\newcommand{\eeqq}{\end{eqnarray*}}
\title{Deep ReLU  neural networks  in high-dimensional approximation}
\author[a]{Dinh D\~ung
	\footnote{Corresponding author: Information Technology Institute, Vietnam National University, Hanoi,
		144 Xuan Thuy, Cau Giay, Hanoi, Vietnam
		\protect\\
		{\it Email address}: dinhzung@gmail.com}
	}
\affil[a]{Information Technology Institute, Vietnam National University, Hanoi
	\protect\\
	144 Xuan Thuy, Cau Giay, Hanoi, Vietnam
\protect\\
Email: dinhzung@gmail.com}

\author[b]{Van Kien Nguyen}
\affil[b]{Faculty of Basic Sciences, University of Transport and Communications
\protect\\	No.3 Cau Giay Street, Lang Thuong Ward, Dong Da District,
	Hanoi, Vietnam
\protect\\
Email: kiennv@utc.edu.vn}

\date{\today}
\begin{document}
\maketitle

\begin{abstract}

 We study  the  computation complexity of  deep  ReLU (Rectified Linear Unit) neural networks  for  the approximation of functions from the H\"older-Zygmund space of mixed smoothness defined on the $d$-dimensional unit cube when the dimension $d$ may be very large.		
The approximation error  is measured in the norm of  isotropic Sobolev space.  For  every function $f$ from the H\"older-Zygmund space of mixed smoothness, we explicitly construct a deep ReLU neural network having an output that  approximates $f$  with a prescribed accuracy $\varepsilon$, and prove tight dimension-dependent  upper and lower bounds of the computation complexity of this approximation, characterized as  the  size and the depth of  this deep ReLU neural network, explicitly in $d$ and $\varepsilon$.  
 The proof of these results are in particular, relied on the approximation by sparse-grid sampling recovery based on  the  Faber series.


  
\end{abstract}

\section{Introduction}
	Neural networks have been studied and used for  almost  80 years, dating back to the foundational works  of McCulloch and  Pitts \cite{MP43},  Hebb \cite{Heb49B} and of Rosenblatt \cite{Ros58}. In recent years, deep neural networks have been successfully applied to a striking variety of Machine Learning problems, including computer vision \cite{KSH12}, natural language processing \cite{WSC.16}, speech recognition and image classification \cite{LBH15}.  The main advantage of deep neural networks over shallow ones is that they can output compositions of functions cheaply. 
	Since their application range is getting wider, theoretical analysis to reveal the reason why deep neural networks could lead to significant practical improvements attracts substantial attention \cite{ ABMM17,DDF.19,MPCB14,Te15,Te16}. 
	In the last several years, there has been a number of interesting papers that address the role of depth and architecture of deep neural networks in approximating sets of functions  which have a very special  regularity properties such as analytic functions \cite{EWa18,Mha96}, differentiable functions \cite{PeVo18,Ya17a}, oscillatory functions \cite{GPEB19}, functions in isotropic Sobolev or Besov spaces \cite{AlNo20,DDF.19,GKNV19,GKP20,Ya17b} and  functions in  spaces of mixed smoothness  \cite{MoDu19, Suzu18}. 
	
	 It has been shown that there is a close relation between  the approximation by sampling recovery based on B-spline interpolation and quasi-interpolation representation, and the  approximation   by     deep neural networks  \cite{AlNo20,DDF.19,MoDu19,Suzu18,ScZe19,Ya17a,Ya17b}.
Most of  these papers used  deep   ReLU  (Rectified Linear Unit)  neural networks for approximation since the rectified linear unit is a simple and preferable activation function in many applications. The output of such a network is a continuous piece-wise linear function which is easily and cheaply  computed. 

In recent decades,  the high-dimensional approximation of functions or signals  depending on a large number $d$ of variables, has been of great interest since they can be applied in a striking number of fields such as Mathematical Finance, Chemistry, Quantum Mechanics, Meteorology, and, in particular, in Uncertainty Quantification and Deep Machine Learning. 
A numerical method for such problems may require a computational cost increasing exponentially in  dimension $d$ when the accuracy increases. This phenomenon is called the curse of dimensionality, coined by Bellman \cite{Bel57B}. Hence
for an efficient computation in  high-dimensional approximation, one of the key prerequisites is that the curse of dimension can be avoided  or eased at least to some extent.  In some cases this can be  achieved, particularly when  the functions to be approximated have an appropriate mixed smoothness, see \cite{BuGr04,NoWo08,NoWo10} and references there. With this restriction one can apply approximation methods and sampling algorithms constructed on  hyperbolic crosses and  sparse grids which  give a surprising effect since hyperbolic crosses and sparse grids have the number of elements much less than those of standard domains and grids but give the same approximation error. This essentially reduces the computational cost, and therefore makes the problem tractable.

The approximation by deep ReLU neural networks of functions having a mixed smoothness is very related to the high-dimensional sparse-grid  approach which was introduced by Zenger for numerical solving partial differential equations (PDEs). 
For functions of mixed smoothness of integer order,  high-dimensional sparse-grid approximations with application was investigated by Bungartz and Griebel \cite{BuGr04}  employing hierarchical Lagrange polynomials multilevel basis  and measuring the approximation error in the  norm of $L_2(\IId)$ or energy norm of  $\mathring{W}_2^1(\IId)$. 
In  the paper \cite{Yser10} on the electronic Schr\"odinger equation with very large number of variables, Yserentant used sparse-grid methods for approximation of the eigenfunctions of the electronic  Schr\"odinger operator having a certain mixed smoothness. 
Triebel \cite[Chapter 6]{Tri15B} has indicated  that when the initial data belongs to spaces with mixed
smoothness,  Navier-Stokes equations admit a unique  solution having some mixed smoothness.
There is a very large number of papers on sparse grids in various problems of high-dimensional approximation in  numerical solving of PDEs and stochastic PDEs, etc. to mention all of them.  The reader can see the surveys in  
\cite{BuGr04} and the references therein.


Consider the problem of approximation of  functions $f$   on $\IId$ in a  space $X$ of a particular smoothness  by trigonometric (for periodic $f$) or dyadic B-splines with accuracy $\varepsilon$ and error measured in the  norm of the space $L_p(\IId)$  or isotropic Sobolev space $W^\gamma_p(\IId)$ $(1\le p\le\infty,\ \gamma >0)$. If  $X$ is   the H\"older space of isotropic smoothness $\alpha$, then the  computation complexity typically is  estimated  similarly by $ C(\alpha,d,p) \varepsilon^{-d/\alpha}$ or $ C(\alpha,d,p) \varepsilon^{-d/(\alpha-\gamma)}$, respectively. For  the H\"older space $X$ of mixed smoothness $\alpha$, the  computation complexity is  bounded by $ C(\alpha,d,p) \varepsilon^{-1/\alpha}\log^{d-1}(\varepsilon^{-1})$ or $ C(\alpha,d,p) \varepsilon^{-1/(\alpha-\gamma)}$, respectively,  i.e., the  bounds of the  computation complexity are of quite different forms. 
Here, $C(\alpha,d,p)$ is a constant depending on $\alpha, d, p$ as well as the norm in which the smoothness is defined.  Similar estimates hold true for the computation complexity of  approximation of functions in Sobolev or Besov type spaces of isotropic and mixed smoothness. Notice also that only in the last case the term in $\varepsilon$ is free from the dimension $d$. As usual, in classical settings of  approximation problem which does not take account of dimension-dependence, this constant is not of interest and its value is not specified.

One of central problems in high-dimensional approximation is to give an evaluation explicit in $d$ for the  term $C(\alpha, d,p)$ 
in the above mentioned estimates of computation complexity to understand the tractability of the approximation problem.  

 We briefly recall  some known results  on approximation by deep ReLU neural networks directly related  to the present paper. 
  	In \cite{Ya17a}, the author constructed a deep ReLU neural network of depth $\cO(\log(\varepsilon^{-1}))$ and size $\cO(\varepsilon^{-d/r} \log(\varepsilon^{-1}))$ that is capable  in $L_\infty(\IId)$-approximating  with accuracy $\varepsilon$ of functions on $\IId$ from the unit ball of isotropic Sobolev space $W^r_\infty(\IId)$ of smoothness $r \in \NN$. By using known results on VC-dimension of deep ReLU neural networks a lower bound also was established for this approximation.  In \cite{GKP20}, this  result was extended to the isotropic Sobolev space $W^r_p(\IId)$ with error measured in the norm of $W^s_p(\IId)$ for $p\in [1,\infty]$ and $s < r$. Considering  the $L_q(\IId)$-approximation of functions from the unit ball of Besov space $B^\alpha_{p,\theta}(\IId)$ of mixed smoothness  $\alpha>\max(0,1/p -1/q)$  by deep ReLU networks  of depth $\cO(\log N)$ and size $ \cO(N\log N)$, the author of \cite{Suzu18} evaluated the approximation error as $\cO( N^{-\alpha}\log^{\alpha(d-1)}N)$. The lower bound of the approximation error was estimated via known results on linear $N$-widths. In \cite{MoDu19}, the authors constructed a deep ReLU  neural network for $L_\infty(\IId)$-approximation  with accuracy $\varepsilon$ of a function with zero boundary condition  in Sobolev space $ X^{p,2} \subset W^2_p(\IId) $ ($p=2,\infty$) of mixed smoothness $2$. Its depth and size are evaluated as  $\cO(\log(\varepsilon^{-1}) \log d)$ and $\cO( \varepsilon^{-1/2} \log^{\frac{3}{2}(d-1) + 1}(\varepsilon^{-1})(d - 1))$, respectively. Notice that all the hidden constants in  these estimates for computation complexity and convergence rate   were not computed explicitly in dimension $d$. In particular, in the proof of the convergence rate in the paper \cite{Suzu18}, the author used a  discrete quasi-norm equivalence of Besov spaces established \cite{Dung11a} which does not allow to find such constants explicit in $d$. Also,  due to the homogeneous  boundary condition  of the functions from the   unit ball of the spaces $ X^{p,2}$ considered in   \cite{MoDu19}, their $d$-dimensional  $L_\infty(\IId)$-norm  is  decreasing as fast as $M_p^{-d}$  for some $M_p > 1$, when $d$ going to infinity, see Remark \ref{rmk4.6} for details.



The purpose of the present paper is to study  the  computation complexity of  deep  ReLU neural networks 
 for the high-dimensional approximation  of functions  from  H\"older-Zygmund  space $\mathring{H}_\infty^\alpha$  of mixed smoothness $\alpha$  satisfying the homogeneous boundary condition, when the dimension $d$ may be very large. 	The approximation error  is measured in the norm of  the  isotropic Sobolev space  $ \mathring{W}_p^1:= \mathring{W}_p^1(\IId)$. We focus our attention on $d$-dependence of this computation complexity.  For every function $f \in \mathring{H}_\infty^\alpha$, we want to explicitly construct a deep ReLU neural network having an output that  approximates $f$  with a prescribed accuracy $\varepsilon$, and prove $d$-dependent   bounds of the computation complexity of this approximation characterized as  its size and the depth,  explicitly in $d$ and $\varepsilon$   (cf. \cite{AnBa09B,DDF.19,Ya17a, MoDu19}). 

	 Let us emphasize that  this problem of approximating functions from the {space} $\mathring{H}_\infty^\alpha$ with error measured in  the norm of  the space $\mathring{W}_p^1$,  in particular,   the energy norm of the space $V:= \mathring{W}_2^1$, naturally arises from some high-dimensional approximation and numerical methods  of PDEs, see  \cite{BuGr99,BuGr04,GGT01,GK09} for Poisson's equation. For elliptic PDEs with homogeneous boundary condition, if  the initial data and diffusion coefficients  have a mixed smoothness, then the solution belongs to  $\mathring{H}_\infty^\alpha$ with a certain $\alpha >0$. One then can consider the problem of approximation of this solution  by deep  ReLU  neural networks with error measured in the energy norm of   $V$. See  a detailed  example  in Remark  \ref{remark:PDE}.

 We briefly describe our contribution  to high-dimensional  approximation by deep ReLU neural networks. {Denote by $\Uas$ the unit ball in the space $\mathring{H}_\infty^\alpha$.}
 For every $f\in \Uas$, we explicitly construct a deep ReLU neural network $\Phi_f$ having  an output $\mathcal{N}(\Phi_f,\cdot)$ that approximates $f$ in  the $ \mathring{W}_p^1$-norm with a prescribed accuracy $\varepsilon$ and having  the computation complexity expressing the dimension-dependent  size 
 \[
 W(\Phi_f ) \leq C_1(\alpha,p) B^{-d}(\varepsilon^{-1})^{\frac{1}{\alpha-1}}\log(\varepsilon^{-1}),
 \]
  and
 the  dimension-dependent depth  
 \[
 L(\Phi_f ) \leq  C_2(\alpha) \log d \log(\varepsilon^{-1}),
 \] 
 where $B=B(d,\alpha,p)>0$.  Notice  the upper bounds of the size $ W(\Phi_f )$ and the depth $ L(\Phi_f )$ consist of three terms. The first term is independent of  the dimension $d$ and  the accuracy $\varepsilon$, the second term depends only on the dimension $d$  and the third term depends only on the accuracy $\varepsilon$. For the depth $ L(\Phi_f )$ the second term $\log d$ is very mild. 
 If a light restriction  holds, in particular when $1\leq p\leq 2$, for the size $ W(\Phi_f )$  the second term $B^{-d}$ satisfies the inequality  $B>1$ when  $d>d_0(\alpha,p)$.

By using a recent result on VC-dimension  bounds for piecewise
linear neural networks in \cite{BHLM19} we prove the following  dimension-dependent lower bound  for the case when $p = \infty$.  For a given $\varepsilon >0$,  if $\mathbb{A}$ is a neural network architecture of depth $L(\mathbb{A})\leq C \log(\varepsilon^{-1})$ such that for any $f\in \Uas$, there is a deep  ReLU neural network $\Phi_f$ of   architecture  $\mathbb{A}$ that approximates $f$ with accuracy $\varepsilon$,	
	 then there exists a constant  $C_3(\alpha)>0$ such that
		\begin{equation*}
		W(\mathbb{A})\geq C_3(\alpha)\, 24^{-\frac{d}{\alpha-1}}\varepsilon^{-\frac{1}{\alpha-1}} (\log (\varepsilon^{-1}))^{- 2}.	
		\end{equation*}		


 The proof of these results, in particular, the construction of the approximating deep ReLU neural networks are relied on  interpolation sampling recovery methods on sparse-grids of points tailored fit to the H\"older-Zygmund mixed smoothness $\alpha$ and the regularity of the isotropic Sobolev space  $\mathring{W}_p^1$. These sampling recovery methods are explicitly constructed as a truncated Faber series of functions to be approximated.

 Let us  analyze some differences  in  the proofs of  results between the present paper and the close paper \cite{Suzu18} as well as  the other related papers \cite{GKP20, MoDu19,Ya17a}(see also Remarks \ref{sparsity}, \ref{rmk3.3},  \ref{rmk4.3} and \ref{rmk4.4}). 
 
 Firstly, to prove the results in \cite{Suzu18}, the author employed  discrete (quasi-)norm equivalence in terms of the valued-functional coefficients of B-spline quasi-interpolation representation for the Besov space $B^\alpha_{p,\theta}(\IId)$ \cite{Dung11a}.  But, as mentioned above, this does not allow to estimate the dimension-dependent component of the approximation error. In the present paper, by using the representation of functions by Faber series we obtained the dimension-dependent bounds for the  size and depth of a deep ReLU neural network required for approximation of functions from $\Uas$. This is a difference in the proofs  between \cite{Suzu18} and the present paper.
 	
	Secondly, in both the papers functions to be approximated have  a certain anisotropic mixed smoothness, but the norm measuring approximation error used in \cite{Suzu18} (also in \cite{ MoDu19}) is of the Lebesgue space $L_q(\IId)$, while in our paper  is of the isotropic space Sobolev $\mathring{W}_p^1(\IId)$. The anisotropic mixed smoothness and the difference between the norms of $L_q(\IId)$ and $\mathring{W}_p^1(\IId)$ together  lead to different methods of construction of (quasi-)interpolation sparse-grid sampling approximation and hence of deep ReLU neural network approximation (notice that these methods  are similar if functions to be approximated have an isotropic  smoothness  \cite{GKP20,Ya17a}). In particular, the authors in \cite{Suzu18} and \cite{ MoDu19} used classical Smolyak grids, while in this paper we use ``notched" Smolyak grids. Therefore, the sparsity of the grid points for interpolation sampling in our paper is much higher than the sparsity of those in \cite{Suzu18} and \cite{ MoDu19}.



The outline of this  paper is as follows. In Section \ref{ReLU}, we recall necessary knowledge of deep ReLU neural networks. 
Section \ref{sec:sampling}  introduces  function spaces under consideration,
presents a  representation of continuous functions on the unit cube $\IId$ by Faber series and  proves some error estimates of approximation by sparse-grid  sampling recovery for functions in H\"older-Zygmund classes $\Uas$. In  Section \ref{sec:dnn}, based on the results in  Section \ref{sec:sampling}, we construct  a deep ReLU neural network that approximates  in the norm of the space $\mathring{W}_p^1$ functions in $\Uas$ and prove  upper and lower estimates for  the size and depth required.   Some concluding remarks are presented in Section \ref{sec:conclusion}.
\\

\noindent
 {\bf Notation.} \ As usual, $\NN$ denotes the natural numbers, $\ZZ$ denotes the integers, $\RR$ the real numbers and $ \NN_0:= \{s \in \ZZ: s \ge 0 \}$; $\NNn= \NN_0\cup \{ -1\} $. 
The letter $d$ is always reserved for
the underlying dimension of $\RR^d$, $\NN^d$, etc., and  $[d]$ denotes the set of all natural numbers from $1$ to $d$. Vectorial quantities are denoted by boldface
letters and  $x_i$ denotes the $i$th coordinate 
of $\bx \in \RR^d$, i.e., $\bx := (x_1,\ldots, x_d)$. We use the notation  $\bx \by $ for 
the usual Euclidean inner product in $\RR^d$ and
$2^\bx := (2^{x_1},\ldots,2^{x_d})$. For $\bk, \bs \in \NNd_0$,  we denote $2^{-\bk}\bs := (2^{-k_1}s_1,\ldots,2^{-k_d}s_d)$. For $\bx\in \RR^d$ we write $|\bx|_0=|\{x_j\not =0, \ j=1,\ldots,d \}|$ and if $0< p\leq \infty$ we denote 
$|\bx|_p := \big(\sum_{i=1}^d |x_i|^p\big)^{1/p}$ with the usual modification when $p=\infty$.
 The notations $|\cdot|_0$ and $|\cdot|_p$  are extended to matrices in $\RR^{m\times n}$.  For the function $f$ on $\RR^d$, $\supp(f)$ denotes the support of $f$. The value $g(\infty)$ of the function $g$ of one variable is understood as $g(\infty)=\lim_{p\to \infty} g(p)$ when the limit exists.

\section{Deep ReLU neural networks} \label{ReLU}

There is a wide variety of  deep neural
network architectures and each of them is adapted to specific tasks.  For approximation of functions from H\"older-Zygmund spaces, in this section we introduce feed-forward  deep ReLU neural networks with one-dimension output. We are interested in  standard deep neural  networks where only connections between
neighboring layers are allowed. Let us introduce necessary definitions and elementary facts on deep ReLU neural networks.

\begin{definition}\label{def:DNN}
	Let $d,L\in \NN$ and $L\geq 2$.
	\begin{itemize} 
		\item A deep neural network  $\Phi$ with input dimension $d$ and $L$ layers  is  a sequence of matrix-vector tuples
		$$\Phi=\big((\bW^1,\bb^1),\ldots,(\bW^L,\bb^L) \big),$$
		where	 $\bW^\ell=(w^\ell_{i,j})$ is an $N_\ell\times N_{\ell-1}$ matrix, and $\bb^\ell =(b^\ell_j)\in \RR^{N_\ell}$ with  $N_0=d$, $N_L=1$, and $N_1,\ldots,N_{L-1}\in \NN$. We call  the number of layers $L(\Phi)=L$  the depth and  $\bN(\Phi)=(N_0,N_1,\ldots,N_L)$ the dimension of the network. The real numbers $w^\ell_{i,j}$ and $b^\ell_j$ are called edge and node weights of the network  $\Phi$, respectively. The number of nonzero weights  $w^\ell_{i,j}$ and $b^\ell_j$  is called    the size of the network $\Phi$ and denoted by $W(\Phi)$, i.e., $
		W(\Phi): =\sum_{\ell=1}^L\big|\bW^\ell \big|_0 + \sum_{\ell=1}^L |\bb^\ell|_0 
		$. We call $N_w(\Phi)=\max_{\ell=0,\ldots,L}\{ N_\ell\}$ the width of the network $\Phi$.
		\item 	A neural network  architecture $\mathbb{A}$ with input dimension $d$ and $L$ layers is a neural network 
		$$\mathbb{A}=\big((\bW^1,\bb^1),\ldots,(\bW^L,\bb^L) \big),$$
		where elements of $\bW^\ell$ and $\bb^\ell$, $\ell=1,\ldots,L$, are in $\{0,1\}$. 
	\end{itemize}
\end{definition}
Since we are interested only in  deep neural networks with scalar output, $\bb^L$ is a constant. However, for consistent notation, we still use bold letter.

A graph associated to a  deep neural network $\Phi$ defined in Definition \ref{def:DNN} is a graph consisting of $|\bN(\Phi)|_1$ nodes and $\sum_{\ell=1}^L|\bW^\ell|_0$ edges.  $|\bN(\Phi)|_1$ nodes are placed in $L+1$ layers which are numbered from $0$ to $L$. The $\ell$th layer has $N_\ell$ nodes which are numbered from 1 to $N_\ell$. If $w^\ell_{i,j}\not =0$, then there is an edge connecting the node $j$ in the layer $\ell-1$ to the node $i$  in the layer $\ell$.
See Figure 	\ref{fig:neuralnetwork} for an illustration of a graph associated to   a deep neural network. 
\begin{definition} 
	Given $L\in \N
	$, $L\geq 2$, and a deep neural network  architecture  $\mathbb{A}= \big((\overline{\bW}^1,\overline{\bb}^1), \ldots,\allowbreak (\overline{\bW}^L,\overline{\bb}^L) \big)$. We say that a neural network  $\Phi=\big((\bW^1,\bb^1),\ldots,(\bW^L,\bb^L) \big)$  has architecture $\mathbb{A}$ if 
	\begin{itemize}
		\item  $\bN(\Phi)=\bN(\mathbb{A})$
		\item   $\overline{w}^\ell_{i,j}=0$ implies $w^\ell_{i,j}=0$, $\overline{b}^\ell_i=0$ implies $b^\ell_i=0$ for all $i=1,\ldots,N_\ell$, $j=1,\ldots,N_{\ell-1}$, and $\ell=1,\ldots, L$.  Here $\overline{w}^\ell_{i,j}$ are entries of $\overline{\bW}^\ell$ and $\overline{b}^\ell_i$ are elements of $\overline{\bb}^\ell$, $\ell=1,\ldots,L$.
	\end{itemize}
	
	For a given deep neural network $\Phi=\big((\bW^1,\bb^1),\ldots,(\bW^L,\bb^L) \big)$, there exists  
	a unique deep neural network  architecture $\mathbb{A}= \big((\overline{\bW}^1,\overline{\bb}^1),\ldots,(\overline{\bW}^L,\overline{\bb}^L) \big)$ such that
	\begin{itemize}
		\item  $\bN(\Phi)=\bN(\mathbb{A})$
		\item   $\overline{w}^\ell_{i,j}=0$ $\Longleftrightarrow$ $w^\ell_{i,j}=0$, $\overline{b}^\ell_i=0$ $\Longleftrightarrow$ $b^\ell_i=0$ for all $i=1,\ldots,N_\ell$, $j=1,\ldots,N_{\ell-1}$, and $\ell=1,\ldots, L$.  
	\end{itemize}
	We call this architecture $\mathbb{A}$ the minimal architecture of $\Phi$ (this definition is proper in the sense that any architecture of $\Phi$ is also an architecture of $\mathbb{A}$.)
\end{definition}
A deep neural network is   associated with an activation function. The choice of activation function depends on the problem under consideration.  In this paper we focus our attention on ReLU activation function defined by 
$\sigma(t):= \max\{t,0\}, t\in \R$.  We will use the notation 
$\sigma(\bx):= (\sigma(x_1),\ldots, \sigma(x_d))$ for $\bx \in \RRd$.
\begin{definition}\label{def:ReLu-network}
	A deep ReLU neural network   with input dimension $d$ and $L$ layers is a neural network  
	$$\Phi=\big((\bW^1,\bb^1),\ldots,(\bW^L,\bb^L) \big)$$
	in which the following
	computation scheme is implemented
	\begin{align*}
	\bz^0&: = \bx \in \RR^d,
	\\
	\bz^\ell &: = \sigma(\bW^{\ell}\bz^{\ell-1}+\bb^\ell), \ \ \ell=1,\ldots,L-1,
	\\
	\bz^L&:= \bW^L\bz^{L-1} + \bb^L.
	\end{align*}
	We call $\bz^0$ the input  and 
	$\mathcal{N}(\Phi, \bx):= \bz^L$ the output of $\Phi$.
\end{definition}

\begin{figure} 
	\begin{center}
		\begin{tikzpicture}  
		\tikzstyle{place}=[circle, draw=black,scale=1, inner sep=3pt,minimum size=1pt, align=center]
		
		\foreach \x in {1,...,3}
		\draw node at (0, -\x*0.8+1.6) [place] (input_\x) { $$};
		
		\foreach \x in {1,...,4}
		\node at (2.5, -\x*0.8+ 2.0) [place] (hidden1_\x){$$};

	\foreach \i in {1,...,3}
\foreach \j in {1,...,4}
\draw [-] (input_\i) to (hidden1_\j);

		\foreach \x in {1,...,5}
		\node at (5, -\x*0.8+2.5) [place] (hidden2_\x){$$};

\foreach \i in {1,...,4}
\foreach \j in {1,...,5}
\draw [-] (hidden1_\i) to (hidden2_\j);
		
		\foreach \x in {1,...,4}
		\node at (7.5, -\x*0.8+2.0) [place] (hiddenl1_\x){$$};

\foreach \i in {1,...,5}
\foreach \j in {1,...,4}
\draw [-] (hidden2_\i) to (hiddenl1_\j);

		\foreach \x in {1,...,4}
		\node at (10, -\x*0.8+2.0) [place] (hiddenl_\x){$$};

\foreach \i in {1,...,4}
\foreach \j in {1,...,4}
\draw [-] (hiddenl1_\i) to (hiddenl_\j);

	\foreach \x in {1}
\node at (12.5, 0) [place] (output_\x){$$};

		\foreach \i in {1,...,4}
		\foreach \j in {1}
		\draw [-] (hiddenl_\i) to (output_\j);

		\node at (0, -2.2) [align=center]{input \\ layer};
		\node at (2.5, -2.2) [align=center] {$1^{st}$\\ layer};
		\node at (5, -2.2) [align=center] {$2^{nd}$\\ layer};
		\node at (7.5, -2.2) [align=center] {$3^{rd}$\\ layer};
		\node at (10, -2.2) [align=center] {$4^{th}$\\ layer};
		\node at (12.5, -2.2) [align=center] {output\\ layer};

		
		\end{tikzpicture}
		\caption{The graph associated to a deep neural network with input dimension 3 and 5 layers}
		\label{fig:neuralnetwork}
	\end{center}
\end{figure}
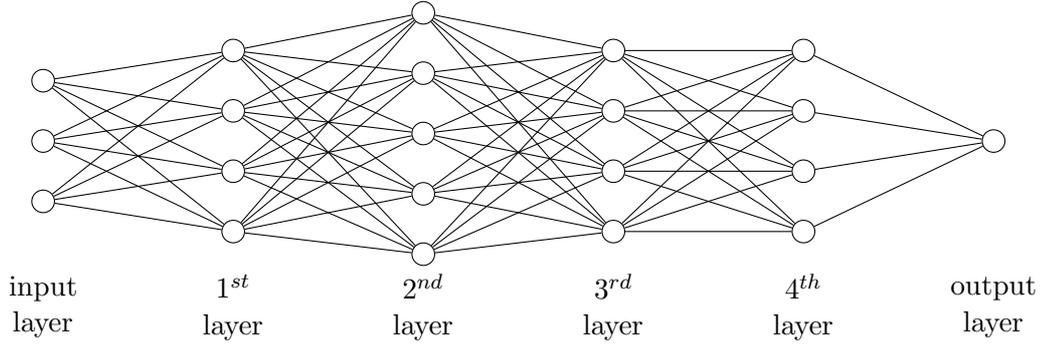

Several deep ReLU neural networks can be combined  with each other to form a larger  deep ReLU neural network whose output is a linear combination of outputs of   the component networks  as in the following lemma. This combination is called parallelization. For other combinations, such as concatenation, we refer to \cite[Section 2]{GKP20} or \cite{DDF.19}. Note that our parallelization construction  below differs slightly from that in \cite{GKP20} since we only consider deep ReLU neural networks  with scalar output. Therefore, for convenience of the reader, we give a proof for this parallelization construction.
\begin{lemma}\label{lem:paralell}
	Let $N\in \NN$, $\Omega\subset \RR^d$ be a bounded set, $\lambda_j\in \RR$, $j=1,\ldots,N$. Let $\Phi_j$, $j=1,\ldots,N$ be deep neural networks with input dimension $d$, $L_j$ layers, and   size $W_j$ respectively. Then we can explicitly construct a deep ReLU neural network $\Phi$ such that 
	$$
	\mathcal{N}(\Phi,\bx)
	=
	\sum_{j=1}^N\lambda_j\mathcal{N}(\Phi_j,\bx),\quad \bx\in \Omega,
	$$ 
	with  $L(\Phi)=\max_{j=1,\ldots,N}\{L_j\}$ and $W(\Phi)=\sum_{j=1}^NW_j+\sum_{j: L_j<L}(L-L_j+2)$. 
\end{lemma}
\begin{proof}We prove first for $N=2$. 
	Without loss of generality we assume that $L_1\leq L_2$ and
	$$\Phi_1=\big((\bW^1_1,\bb^1_1),\ldots, (\bW^{L_1}_1,\bb^{L_1}_1) \big);\qquad \Phi_2=\big((\bW^1_2,\bb^1_2),\ldots, (\bW^{L_2}_2,\bb^{L_2}_2) \big).$$
	If $L=L_1= L_2$, then we can choose
	$$
	\Phi=\big((\bW^1,\bb^1),\ldots, (\bW^{L},\bb^{L}) \big),
	$$
	where
	$$\bW^1=\begin{bmatrix} 
	\bW^1_1 \\
	\bW^1_2\\
	\end{bmatrix},\ \bW^\ell=\begin{bmatrix} 
	\bW^{\ell}_1 & 0 \\
	0 & \bW^{\ell}_2\\
	\end{bmatrix}, \ \ell=2,\ldots,L-1, \  \bW^{L}=\begin{bmatrix} 
	\lambda_1 \bW^{L}_1  & \lambda_2 \bW^{L}_2\\
	\end{bmatrix}
	$$
	and
	$$\bb^1=\begin{bmatrix} 
	\bb^1_1 \\
	\bb^1_2\\
	\end{bmatrix}, \ \ell=1,\ldots,L_1
	,\quad 
	\bb^{L}= 
	\lambda_1	\bb^{L}_1  +
	\lambda_2	\bb^{L}_2 .
	$$ 
	In this case we have $W(\Phi)\leq W_1+W_2$. If $L_1< L_2$ we construct a network $\tilde{\Phi}_1$ with output  $\mathcal{N}(\tilde{\Phi},\bx)=\mathcal{N}(\Phi,\bx)$ and having $L_2$ layers. The strategy here is to modify the network $\Phi_1$ by making its output layer satisfying $\mathcal{N}(\Phi,\bx)+C\geq 0$, for some constant $C$, so that this value does not change when we apply function $\sigma$ for layers from $L_1+1$ to $L_2$. For this we put $M_1=\sup_{\bx\in \Omega}|\mathcal{N}(\Phi_1,\bx)|$. Note that $\mathcal{N}(\Phi_1,\cdot)$ is a continuous function on $\Omega$ hence $M_1<\infty$.  The network $\tilde{\Phi}_1$ is  
	$$\big((\bW^1_1,\bb^1_1),\ldots, (\bW^{L_1-1}_1,\bb^{L_1-1}_1),(\bW^{L_1}_1,\bb^{L_1}_1+M_1), (1,0),\ldots, (1,0), (1,-M_1) \big).$$
	Hence $W(\tilde{\Phi}_1)\leq W_1+L_2-L_1+2 $.
	Now following procedure as the case $L_1=L_2$ with $\Phi_1$ replaced by $\tilde{\Phi}_1$ we obtain the assertion when $N=2$. The case $N>2$ is extended in a similar manner. 
	\hfill \end{proof}
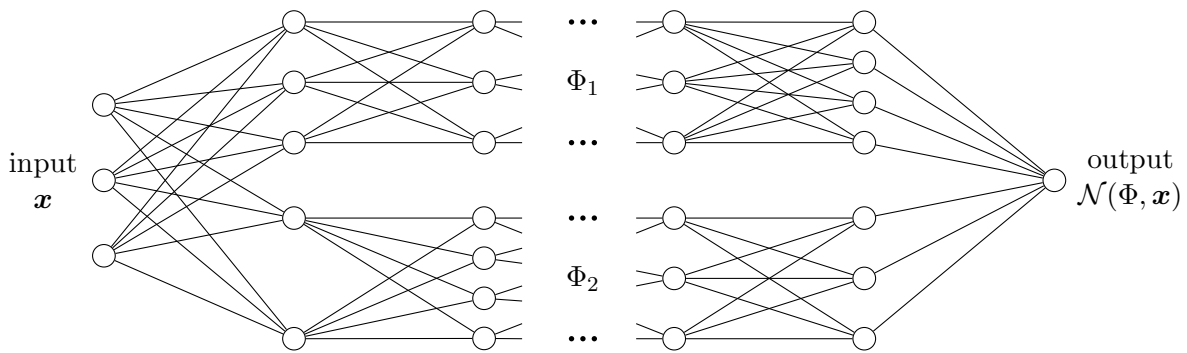
\begin{figure}
	\begin{center}
		\begin{tikzpicture}
		\tikzstyle{place}=[circle, draw=black,scale=1, inner sep=3pt,minimum size=1pt, align=center]
		
		\foreach \x in {1,...,3}
		\draw node at (0, -\x*1.0+2.0) [place] (input_\x) { $$};
		\node at (-0.8, 0) [align=center]{input
			\\ $\bx$};
		
		\foreach \x in {1,...,3}
		\node at (2.5, 0.5 + \x*0.8+ -0.8) [place] (hiddenks11_\x){$$};
		
		\foreach \x in {1,...,2}
		\node at (2.5,  -0.5-\x*1.6 + 1.6) [place] (hiddenks21_\x){$$};

		\foreach \x in {1,...,3}
		\node at (5.0, 0.5 +\x*0.8-0.8) [place] (hiddenks12_\x){$$};

		\foreach \x in {1,...,4}
		\node at (5.0, -0.5 -\x*0.532+0.532) [place] (hiddenks22_\x){$$};

\foreach \x in {1,...,3}
\fill (6.0+\x*0.15, 2.1) circle (1pt);
\foreach \x in {1,...,3}	
\fill (6.0+\x*0.15, 0.5) circle (1pt);		
\foreach \x in {1,...,3}
\fill (6.0+\x*0.15, -2.1) circle (1pt);
\foreach \x in {1,...,3}	
\fill (6.0+\x*0.15, -0.5) circle (1pt);

		\foreach \x in {1,...,3}
		\node at (7.5, 0.5+\x*0.8- 0.8) [place] (hiddenks13_\x){$$};

		\foreach \x in {1,...,3}
		\node at (7.5, -0.5-\x*0.8+0.8) [place] (hiddenks23_\x){$$};
	
	\foreach \x in {1,...,4}
	\node at (10, 0.5 + \x*0.532-0.532) [place] (hiddenks14_\x){$$};
	
	\foreach \x in {1,...,3}
	\node at (10,  -0.5-\x*0.8+0.8) [place] (hiddenks24_\x){$$};
		
		\node at (12.5, 0)  [place] (output){$$};
		\node at (13.5, 0) [align=center]{output
		\\
	 $\mathcal{N}(\Phi,\bx)$};


\foreach \i in {1,...,3}
\foreach \j in {1,...,3}
\draw [-] (input_\i) to (hiddenks11_\j);

\foreach \i in {1,...,3}
\foreach \j in {1,...,2}
\draw [-] (input_\i) to (hiddenks21_\j);


\foreach \i in {1,...,3}
\foreach \j in {1,...,3}
\draw [-] (hiddenks11_\i) to (hiddenks12_\j);

\foreach \i in {1,...,2}
\foreach \j in {1,...,4}
\draw [-] (hiddenks21_\i) to (hiddenks22_\j);


\foreach \i in {1,...,3}
\foreach \j in {1,...,4}
\draw [-] (hiddenks13_\i) to (hiddenks14_\j);

\foreach \i in {1,...,3}
\foreach \j in {1,...,3}
\draw [-] (hiddenks23_\i) to (hiddenks24_\j);


\foreach \i in {1,...,4}
\draw [-] (hiddenks14_\i) to (output);

\foreach \i in {1,...,3}
\draw [-] (hiddenks24_\i) to (output);



\foreach \x in {1,2}
\draw [-] (hiddenks12_1) to (5.5, 0.5+0.2*\x -0.2);
\foreach \x in {1,2}
\draw [-] (hiddenks12_2) to (5.5, 1.4-0.2*\x+0.2);
\foreach \x in {1,2}
\draw [-] (hiddenks12_3) to (5.5, 2.1-0.2*\x +0.2);

\foreach \x in {1,2}
\draw [-] (hiddenks22_4) to (5.5, -2.3+0.2*\x);
\foreach \x in {1,2}
\draw [-] (hiddenks22_3) to (5.5, -1.8+0.2*\x);
\foreach \x in {1,2}
\draw [-] (hiddenks22_2) to (5.5, -1.3+0.2*\x);
\foreach \x in {1,2}
\draw [-] (hiddenks22_1) to (5.5, -0.9+0.2*\x);

\node at (6.3, 1.3) [align=center]{$\Phi_1$};
\node at (6.3, -1.3) [align=center]{$\Phi_2$};

\foreach \x in {1,2}
\draw [-] (7.0, 2.1-0.2*\x+0.2) to (hiddenks13_3);
\foreach \x in {1,2}
\draw [-]  (7.0, 1.6-0.2*\x) to (hiddenks13_2);
\foreach \x in {1,2}
\draw [-]  (7.0, 0.5+0.2*\x-0.2) to (hiddenks13_1);

\foreach \x in {1,2}
\draw [-] (7.0, -2.1+0.2*\x-0.2) to (hiddenks23_3);
\foreach \x in {1,2}
\draw [-]  (7.0, - 1.6+0.2*\x) to (hiddenks23_2);
\foreach \x in {1,2}
\draw [-]  (7.0, -0.5-0.2*\x+0.2) to (hiddenks23_1);
		\end{tikzpicture}
		\caption{The graph associated to parallelization of two neural networks}
		\label{fig:parallelization}
	\end{center}
\end{figure}

An illustration of parallelization of neural networks is given in Figure \ref{fig:parallelization}.

Let us introduce a concept of special deep neural network borrowed from  \cite{DDF.19}. It is quite  useful in construction deep ReLU neural networks with a fixed width, whose outputs are able to approximate  multivariate functions.

\begin{definition}
	A special deep neural network with  input dimension $d$ and depth $L$  (and a given activation function) can be defined as follows.  In each hidden layer a special role is reserved for $d$ first (top) nodes and the last (bottom) node. The top $d$ nodes and the bottom node are free of the activation function, other nodes in each hidden layer have the activation function. The top $d$ nodes are used to simply copy the input $\bx$. The $d$ parallel concatenations of all these top nodes can be viewed as  special channels that skip computation altogether and just carry $\bx$ forward. They are called the source channels. The bottom node in each hidden lawyer is used to collect intermediate outputs by addition. The concatenation of all these nodes is called collation channel. This channel never feeds forward into subsequent calculation, it only accepts previous calculations. 
\end{definition}

An illustration of a special deep neural network is given in Figure \ref{fig:special}. 

\begin{lemma}\label{lem:special}
	Let  $\Phi$ be a special deep ReLU neural network with input dimension $d$ and depth $L$.  Then there is a deep ReLU neural network  $\Phi'$ such that, $N_w(\Phi') = N_w(\Phi)$, $L(\Phi') = L$, and $\mathcal{N}(\Phi',\bx) = \mathcal{N}(\Phi,\bx)$, $\bx \in \IId$.
\end{lemma}
\begin{proof}
	The proof follows from  \cite[Remark 3.1]{DDF.19}. First note, that the input $\bx$ belongs to $\IId = [0,1]^d$, we have $\bx = \sigma(\bx)$. For $\ell=1,\ldots,L$, the bottom node in the $\ell$-th layer collects a continuous piece-wise linear function $g_\ell(\bx)$ on $\IId$, and the output is
	$\mathcal{N}(\Phi,\bx) = \sum_{\ell=1}^{L-1} g_\ell(\bx)$.  Thus, there is a constant $c_\ell$ such that $g_\ell(\bx) + c_\ell  \ge 0$ for all $\bx \in \IId$. Hence we take $\Phi'$ having the same graph as $\Phi$ but the computation at $\ell$-th nodes in the collation channel is replaced by 
	$\sigma(g_\ell(\bx) + c_\ell)=g_\ell(\bx) + c_\ell$, 
	and the output node is
	$$\mathcal{N}(\Phi',\bx) = 
	\sum_{\ell=1}^{L-1} \sigma(g_\ell(\bx) + c_\ell) - \sum_{\ell=1}^{L-1}c_\ell=\mathcal{N}(\Phi,\bx) .$$
	\hfill
\end{proof}	

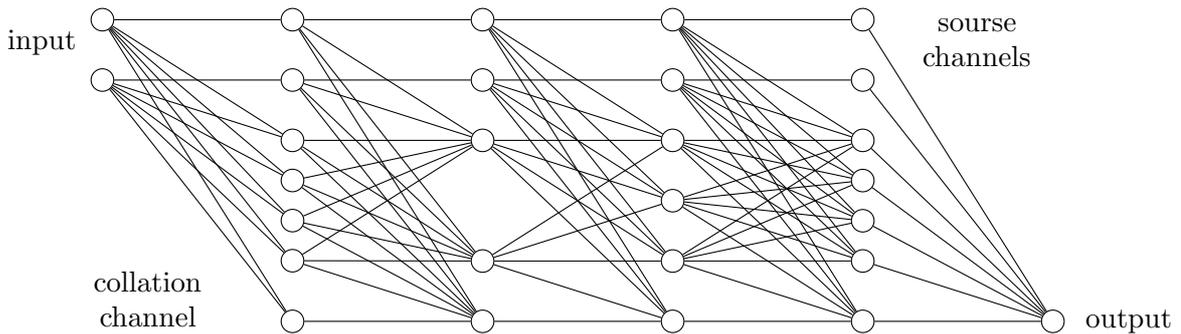
\begin{figure}
\begin{center}
\begin{tikzpicture}
\tikzstyle{place}=[circle, draw=black,scale=1, inner sep=3pt,minimum size=1pt, align=center]

\foreach \x in {1,...,2}
\draw node at (0.0, -\x*0.8 +0.8) [place] (input_\x) { $$};
\node at (-0.8, -0.3) [align=center] {input};
\node at (0.6, -3.7) [align=center] {collation\\ channel};
\node at (13.5, -4.) [align=center] {output};
\node at (11.5, -0.3) [align=center] {sourse\\ channels};

\foreach \x in {1,...,2}
\node at (2.5, -\x*0.8 +0.8) [place] (hiddenx1_\x){$$};

\foreach \x in {1,...,4}
\node at (2.5, -1.6 -\x*0.532+0.532) [place] (hiddenks1_\x){$$};

\node at (2.5, -4.0) [place] (hiddencol1){$$};

\foreach \x in {1,...,2}
\node at (5.0,  -\x*0.8 +0.8) [place] (hiddenx2_\x){$$};

\foreach \x in {1,...,2}
\node at (5.0, -1.6 -\x*1.6+1.6) [place] (hiddenks2_\x){$$};

\node at (5.0, -4.0) [place] (hiddencol2){$$};

\foreach \x in {1,...,2}
\node at (7.5, -\x*0.8 +0.8) [place] (hiddenx3_\x){$$};
\foreach \x in {1,...,3}
\node at (7.5, -1.6 -\x*0.8+0.8) [place] (hiddenks3_\x){$$};
\node at (7.5, -4.0) [place] (hiddencol3){$$};

\foreach \x in {1,...,2}
\node at (10.0,  -\x*0.8 +0.8) [place] (hiddenx4_\x){$$};
\foreach \x in {1,...,4}
\node at (10.0, -1.6 -\x*0.533+0.533) [place] (hiddenks4_\x){$$};
\node at (10.0, -4.0) [place] (hiddencol4){$$};


\node at (12.5, -4.0) [place] (output){$$};





\foreach \i in {1,...,2}
\foreach \j in {1,...,4}
\draw [-] (input_\i) to (hiddenks1_\j);
\foreach \i in {1,...,2}
\draw [-] (input_\i) to (hiddencol1);
\foreach \i in {1,...,2}
\draw [-] (input_\i) to (hiddenx1_\i);


\foreach \i in {1,...,2}
\draw [-] (hiddenx1_\i) to (hiddenx2_\i);

\foreach \i in {1,...,4}
\foreach \j in {1,...,2}
\draw [-] (hiddenks1_\i) to (hiddenks2_\j);

\foreach \i in {1,...,2}
\foreach \j in {1,...,2}
\draw [-] (hiddenx1_\i) to (hiddenks2_\j);

\foreach \i in {1,...,2}
\draw [-] (hiddenx1_\i) to (hiddencol2);

\foreach \i in {1,...,4}
\draw [-] (hiddenks1_\i) to (hiddencol2);

\draw [-] (hiddencol1) to (hiddencol2);


\foreach \i in {1,...,2}
\draw [-] (hiddenx2_\i) to (hiddenx3_\i);

\foreach \i in {1,...,2}
\foreach \j in {1,...,3}
\draw [-] (hiddenks2_\i) to (hiddenks3_\j);

\foreach \i in {1,...,2}
\foreach \j in {1,...,3}
\draw [-] (hiddenx2_\i) to (hiddenks3_\j);

\foreach \i in {1,...,2}
\draw [-] (hiddenx2_\i) to (hiddencol3);

\foreach \i in {1,...,2}
\draw [-] (hiddenks2_\i) to (hiddencol3);

\draw [-] (hiddencol2) to (hiddencol3);


\foreach \i in {1,...,2}
\draw [-] (hiddenx3_\i) to (hiddenx4_\i);

\foreach \i in {1,...,3}
\foreach \j in {1,...,4}
\draw [-] (hiddenks3_\i) to (hiddenks4_\j);

\foreach \i in {1,...,2}
\foreach \j in {1,...,4}
\draw [-] (hiddenx3_\i) to (hiddenks4_\j);

\foreach \i in {1,...,2}
\draw [-] (hiddenx3_\i) to (hiddencol4);

\foreach \i in {1,...,3}
\draw [-] (hiddenks3_\i) to (hiddencol4);

\draw [-] (hiddencol3) to (hiddencol4);


\foreach \i in {1,...,2}
\draw [-] (hiddenx4_\i) to (output);

\foreach \i in {1,...,4}
\draw [-] (hiddenks4_\i) to (output);

\draw [-] (hiddencol4) to (output);
\end{tikzpicture}
\caption{The graph associated to a special neural network with two source channels and 5 layers}
\label{fig:special}
\end{center}
\end{figure}

\section{ Faber series and high-dimensional sparse-grid sampling recovery}\label{sec:sampling}

In this  section  we  introduce the space $\Ha$  of  functions having  H\"older-Zygmund  mixed smoothness $\alpha >0$, and the  isotropic Sobolev space  $ \mathring{W}_p^1$;
 recall a representation of continuous functions on $\IId$  by tensor product Faber series.
This representation plays a fundamental role in construction of sparse-grid sampling recovery and of deep neural networks for approximation in  the $ \mathring{W}_p^1$-norm of functions from the  space  $\Ha$.  We explicitly  construct linear sampling  methods on sparse grids $R_\beta(m,\cdot)$ and  prove some  estimate  explicit in $d$ and $m$ of the error of the approximation by these sampling operators.


\subsection{Function spaces}\label{sec-sub-function}

Unlike the univariate Sobolev,  H\"older-Zygmund and Besov spaces which have   similar approximation properties, the multivariate Sobolev, H\"older and Besov spaces of mixed smoothness have very different approximation properties. Hence methods of approximation based on hyperbolic crosses and sparse grids (in particular, the choice of proper hyperbolic crosses and sparse grids) and convergence rates of the approximation error and computation complexities of  approximation of functions from these spaces are also different.
We refer the reader to \cite{Tem18B,Dung11a,Dung11b,Dung16,DTU18B}    for  surveys and bibliography on various aspects of approximation of functions having mixed smoothness and applications.

In this  subsection we introduce the H\"older-Zygmund  space $\Ha$ of functions having mixed smoothness $\alpha >0$  for $0 < \alpha \le 2$. There are several definitions of  H\"older-Zygmund spaces which are currently in use. 
These definitions are equivalent with certain restrictions on the parameters. Traditional approximation problems where the dimension $d$ (the number of variables) is small and fixed the convergence rate of approximation error and the computation complexity with respect to different equivalent  norms differ by only moderate constants. The picture completely changes for high-dimensional approximation problems when we stress the accurate $d$-dependence in evaluation of these quantities. In fact, it is essentially depends on the choice of a norm defining a class of functions to be approximated and a norm measuring the approximation error (cf. \cite{NoWo08,DTU18B,DU13,KSU15,ChD16} for detailed discussions).  In this paper, we introduce  H\"older-Zygmund spaces by using  differences.   This definition has been used by many authors,  see, \cite{DTU18B,Ni75B,Tem93B,Tem18B} for details and bibliography.

For a univariate function $f$ on $\II:=[0,1]$, the  second difference operator $\Delta_h^2$ is defined by 
\begin{equation*}
\Delta_h^2(f,x) :=   f(x + 2h)-2f(x+h)+f(x),
\end{equation*}
for all $x$ and $h \ge 0$ such that $x, x +2h \in \II$.
If $u$ is any subset of $[d]:=\{ 1,\ldots,d\}$, for a multivariate function $f$ on $\IId$
the mixed $(2,u)$th difference operator $\Delta_\bh^{2,u}$ is defined by 
\begin{equation*}
\Delta_\bh^{2,u} := \
\prod_{i \in u} \Delta_{h_i}^2, \quad \Delta_\bh^{2,\varnothing} = {\rm Id},
\end{equation*}
where the univariate operator
$\Delta_{h_i}^2$ is applied to the univariate function $f$ by considering $f$ as a 
function of  variable $x_i$ with the other variables held fixed, and {\rm Id} is the identity operator.

	 Let $0<\alpha\leq 2$ be given.  The H\"older-Zygmund space 
		$\Ha$ of mixed smoothness $\alpha$ is defined as the set of  all functions $f \in C(\IId)$ 
		for which the  norm 
	\begin{equation} \label{eq:def-norm}
		\|f\|_{\Ha}
		:= \ 
		\max_{u \subset [d]} \bigg\{\sup_{\bh} \ \prod_{i \in u} h_i^{-\alpha}\|\Delta_\bh^{2,u}(f)\|_{C(\IId(\bh,u))}\bigg\}
	\end{equation}
	is finite, where 
	$\IId(\bh,u):= \{\bx \in \IId: \, x_i + 2h_i \in \II, \, i \in u\}$  and for a subset $\Omega$ in $\IId$ and a continuous function $f$ on $\IId$ the norm $\|f\|_{C(\Omega)}$ is defined by 
		\begin{equation*} \label{C-norm}
		\|f\|_{C(\Omega)}
		:= \ 
		\sup_{\bx \in \Omega} |f(\bx)|. 
		\end{equation*}
	Note, that when $u=\varnothing$ the term in brackets of \eqref{eq:def-norm} is $ \|f\|_{C(\IId)}$.
	 We say that a  function  $f$ on $(\IId)$ satisfies the homogeneous boundary condition if $f(\bx) =0$ if  $x_j=0$ or $x_j=1$  for some index $j \in [d]$, i. e., $f$ vanishes in the boundary of the cube $\IId$.
	We define by ${\mathring{H}_\infty^\alpha}$ the subspace 
	of all functions $f$ in  $\Ha$ satisfying the homogeneous boundary condition. 

By the definition we have the inclusions $\Ha \subset C(\IId) \subset L_p(\IId)$ for $0 <p \le \infty$. Moreover, it is well-known that  for  $\alpha=1,2$, the space $H_\infty^\alpha(\II)$  coincides with the Sobolev space $W_\infty^\alpha(\II)$. Hence, for  $\alpha=1,2$ by a tensor product argument one can deduce that the space $\Ha $ coincides with the space of all functions $f\in L_\infty(\IId)$ such that   mixed derivatives $\partial^{\bbeta}f$ with $\bbeta\in \NN_0^d$, $|\bbeta|_\infty \leq \alpha$,  belong  to $L_\infty(\IId)$.   For further properties of these spaces such as embeddings, characterization by wavelets, atoms or B-splines, we refer the reader to \cite{Ni75B,ST87B,Vy06,Dung16, DTU18B} and references there.

Since the results in \cite{MoDu19,Suzu18,Ya17b,GKP20} are closely related to  the present paper, we make some comments on $\Ha$ and spaces defined in these papers. First of all, it is clear that the isotropic classes considered in \cite{Ya17b,GKP20} are totally different from the classes of mixed smoothness in this paper and in  \cite{MoDu19,Suzu18}, except for $d=1$. 		
 Notice that the space  ${\mathring{H}_\infty^2}(\IId)$ coincides with the space $X^{\infty,2}$ of mixed smoothness $2$ which was considered  in \cite{MoDu19}. 	In case $\alpha\not \in \NN$, the space $\Ha$ coincides with the Besov space of mixed smoothness $B^\alpha_{\infty,\infty}(\IId)$  in \cite{Suzu18}.

%
%

In the present paper, we investigate the   deep  ReLU neural network  approximation  of functions  from  H\"older-Zygmund  space $\Ha$ of mixed smoothness $\alpha$, when the dimension $d$ may be very large. 	The approximation error  is measured in the norm of  the  isotropic Sobolev space  $ \mathring{W}_p^1:=\mathring{W}_p^1(\IId)$.  The  space  ${\mathring{W}_p^1}$, $1\leq p\leq \infty$, consists of all functions $f\in L_p(\IId)$   satisfying the homogeneous boundary condition (in the sense of trace) such that  the  norm
	\begin{equation*} 
		\|f\|_{{ \mathring{W}_p^1}}: =
		\begin{cases}
			\bigg( \displaystyle\sum_{i=1}^d \int_{\IId}\bigg|\frac{\partial}{\partial x_i} f(\bx)\bigg|^p \dd \bx\bigg)^{1/p}, & 1\leq p < \infty, \\
			\displaystyle\max_{1 \le i \le d}\, \esssup_{\bx \in \IId}  \Big|\frac{\partial}{\partial x_i} f(\bx)\Big|, & p = \infty,
		\end{cases}
	\end{equation*}
	is finite (this is a norm due to the Poincar\'e inequality).  
	
%

\subsection{High-dimensional sparse-grid sampling recovery}\label{subsec-sparse-grid}

We start  with introducing the tensorized Faber  basis.
Let  $M_2(x)\ :=  \max (0, 1 - |x-1|)$, $x \in  \RR$, be  the hat function (the piece-wise linear B-spline with knots at $0,1,2$).
For $k\in \NNn$ we define the functions $\varphi_{k,s}$  on $\II$  by
\begin{equation*}\label{eq:faber1}
\varphi_{k,s}(x):=
M_2(2^{k+1}x - 2s), \  x \in \II, \quad k \geq 0,  \ s \in Z(k):=\{0,1,\ldots, 2^{k} - 1\},
\end{equation*}
and
\begin{equation*}\label{eq:faber2}
\varphi_{-1,s}(x) := M_2(x - s + 1),\  x \in \II, \quad  s\in Z(-1):=\{0,1\}.
\end{equation*}
	The system
$
\big\{\varphi_{k,s}: k \in \NNn,\, s\in Z(k)\big\}
$ is the classical Faber basis for $C(\II)$, see, e. g., \cite[Section 5.5]{DTU18B}.

Put $Z^d(\bk):={\mathlarger{\mathlarger{\mathlarger{\mathlarger{\times}}}}}_{i=1}^d Z(k_i)$. 
For $\bk \in \NNdn$, $\bs \in Z^d(\bk)$, define the $d$-variate  tensor product hat functions
\begin{equation} \label{hat-function}
\varphi_{\bk,\bs}(\bx)
\ := \
\prod_{i=1}^d \varphi_{k_i,s_i}(x_i),\quad \bx\in \IId.
\end{equation}

Notice that the hat function  $M_2$ can be represented as a linear combination of three ReLU functions:
	\[
	M_2(x) \ = \ \sigma (x) - 2\sigma(x - 1) + \sigma (x-2).
	\]
This explains why the tensor product hat functions $\varphi_{\bk,\bs}$ are extremely well-suited to approximation by deep ReLU neural networks (see Lemma \ref{lem:product} below).	{It is worth also to mention that the univariate Faber basis isolates ``teeth" of the sawtooth functions $f_{{\rm m}}^k$ used  in study of benefits of deep ReLU networks  in \cite{Te15}.}

For a univariate function $f$ on $\II$, 
$k \in \NNn$, and $s\in Z(k)$ we define
\begin{equation*} 
\lambda_{k,s}(f) \ := 
- \frac {1}{2} \Delta_{2^{-k-1}}^2 f\big(2^{-k}s\big), \ k \ge 0, \quad 
\lambda_{-1,s}(f) \ := f(s).
\end{equation*}
We also define the linear functionals $\lambda_{\bk,\bs}$ for multivariate function $f$ on $\IId$, $\bk\in \NNdn$, and $\bs\in Z^d(\bk)$ by 
\begin{equation*} 
\lambda_{\bk,\bs}(f) \ := 
\prod_{i=1}^d \lambda_{k_i,s_i}(f),
\end{equation*}
where the univariate functional $\lambda_{k_i,s_i}$ is applied to the univariate function $f$ by considering $f$ as a function of variable $x_i$ with the other variables held fixed.

We have the following  representation by Faber series for continuous functions on $\IId$ (see \cite[Section 4]{Dung11a} and \cite[Theorem 3.10]{Tri10B}).
\begin{lemma} \label{lemma[convergence(d)]}
	The Faber system
	$
	\big\{\varphi_{\bk,\bs}: \, \bk \in \NNdn,\, \bs\in Z^d(\bk)\big\}
	$ is a basis in $C(\IId)$. Moreover, any function 
	$f \in C(\IId)$ can be represented by the Faber series 
	\begin{equation*} \label{eq:FaberRepresentation}
	f
	\ = \
	\sum_{\bk \in \NNdn} q_\bk(f) 
	:= \
	\sum_{\bk \in \NNdn} \sum_{\bs\in Z^d(\bk)} \lambda_{\bk,\bs}(f)\varphi_{\bk,\bs}, 
	\end{equation*}
	converging in the norm of $C(\IId)$.
\end{lemma}


Denote by  $\Uas$ the unit ball in ${\mathring{H}_\infty^\alpha}$.  If $f \in \Uas$, we can write
\begin{equation*}
f \ = \ \sum_{\bk \in \NNd_0} q_\bk(f)   
\end{equation*}
with unconditional convergence in $C(\IId)$, see \cite[Theorem 3.13]{Tri10B}.
In this case we have
\begin{equation*} 
\lambda_{\bk,\bs}(f) \ := \prod_{i = 1}^d
\bigg(- \frac {1}{2} \Delta_{2^{-k_i-1}}^2 f\big(2^{-\bk}\bs\big)\bigg)
\end{equation*}
and  by \eqref{eq:def-norm} it holds the following estimate
\begin{equation}\label{eq:estimate}
|\lambda_{\bk,\bs}(f)|\leq 2^{ -(\alpha + 1)d}2^{- \alpha|\bk|_1},\quad \bk\in \NNd_0, \ \bs\in Z^d(\bk).
\end{equation}

We now construct sparse grids and sampling operators on them for approximately recovering  functions in $\Uas$ from their values on these grids. For $\beta\geq 1$ and $m\in \N$, we define the sets of multi-indices
\begin{equation*}
\Delta^d_\beta(m) : =\big\{\bk\in \NNd_0:\, |\bk|_1=m-j,\ |\bk|_\infty\geq m-\lfloor\beta j\rfloor\quad \text{for} \quad j=0,\ldots ,m\big\}, 
\end{equation*}   
and 
\begin{equation}  \nonumber
D^d_\beta(m)
:= \big\{ (\bk,\bs): \bk\in \Delta^d_\beta(m), \ \bs\in Z^d(\bk)\big\}.
\end{equation}
The definition of $\Delta^d_\beta(m)$ is  similar to  $ \Delta(\alpha,\beta;\xi)=\big\{\bk\in \NNd_0: \ \alpha |\bk|_1-\beta|\bk|_\infty \leq \xi\big\}$ introduced in \cite{BDSU16} but simpler. 
We also put 
\begin{equation*}
\Delta^d(m) : =\big\{\bk\in \NNd_0:  |\bk|_1  \leq m \big\}
\end{equation*}
which corresponds to the well-known Smolyak grid. It is obvious that $\Delta^d_\beta(m)$ is a subset of $ \Delta^d(m)$ for all $\beta\geq 1$.

\begin{figure}
	
	\begin{tabular}{ccc}
			\begin{tikzpicture}
	\foreach \x in {0,...,20}
	\foreach \y in {0,...,20}
	\fill (\x*0.2,  \y*0.2) circle (2pt);
	\draw [->,  thick] (0,0) to (4.4,0);
	\draw [->,  thick] (0,0) to (0,4.4);
	\node at (4.3, -0.4) {$k_1$};
	\node at (-0.4,4.2) {$k_2$};
	\end{tikzpicture}
	&

		\begin{tikzpicture}
\foreach \m in {0,...,20}
\foreach \x in {0,...,\m}
\fill (\x*0.2, \m*0.2-\x*0.2) circle (2pt);
	\draw [->,  thick] (0,0) to (4.4,0);
\draw [->,  thick] (0,0) to (0,4.4);
\node at (4.3, -0.4) {$k_1$};
\node at (-0.4,4.2) {$k_2$};
	\end{tikzpicture}
	&
\begin{tikzpicture}

\foreach \j in {0,...,7}
\foreach \y in {0,...,\j}{
	\fill (\y*0.2, 20*0.2-\j*0.2-\y*0.2) circle (2pt);
	\fill ( 20*0.2-\j*0.2-\y*0.2, \y*0.2) circle (2pt);
}
\foreach \m in {0,...,13}
\foreach \x in {0,...,\m}
\fill (\x*0.2, \m*0.2-\x*0.2) circle (2pt);
	\draw [->,  thick] (0,0) to (4.4,0);
\draw [->,  thick] (0,0) to (0,4.4);
\node at (4.3, -0.4) {$k_1$};
\node at (-0.4,4.2) {$k_2$};
\end{tikzpicture}
\end{tabular}
\caption{ Illustration of different sets of multi-indices in $\NN_0^2$. The left graph is  $\{\bk \in \NN_0^2: |\bk|_\infty \leq 20\}$, the  middle is $\Delta^2(20)$ and the right  is $\Delta_{2}^2(20)$.}
		\label{fig:grids}

\end{figure}
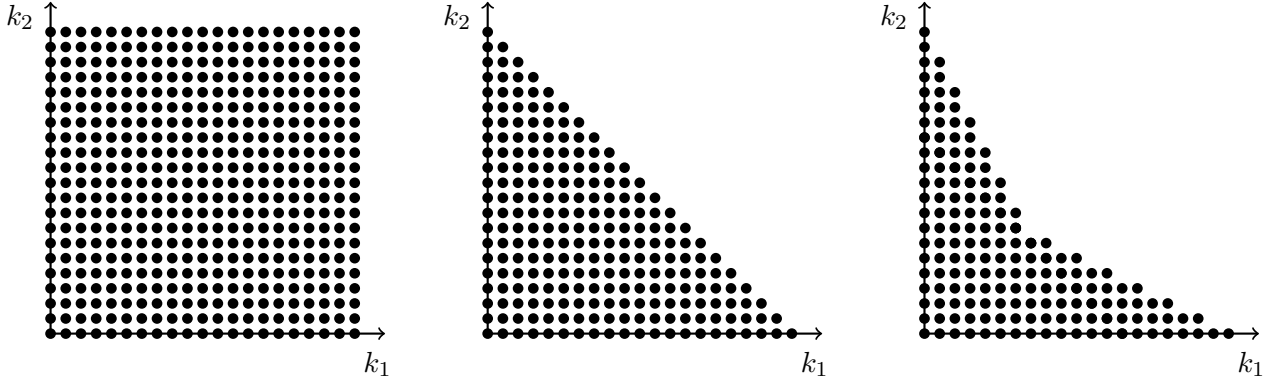
Consider the  operator
\begin{equation}\label{Rbm}
R_{\beta}(m,f) 
:= 
\sum_{\bk \in \Delta^d_\beta(m)} q_\bk(f)= 
\sum_{\bk \in \Delta^d_\beta(m)} \ \sum_{\bs\in Z^d(\bk)} \lambda_{\bk,\bs}(f)\varphi_{\bk,\bs},
\end{equation}
defined for $f \in \Uas$, and the energy-norm-based grid
\begin{equation}  \nonumber
G^d_\beta(m)
:= \big\{ 2^{-\bk}\bs:\ \bk\in \Delta^d_\beta(m) ,\ \bs  \in Z^d(\bk +\boned)\},\qquad 
\boned :=(1,\ldots,1)\in \RR^d.
\end{equation}


 The operator $R_{\beta}(m,\cdot)$  represents a sampling recovery method for functions from $\Uas$.  We notice some important properties of the operator $R_{\beta}(m,\cdot)$ and the grid $G^d_\beta(m)$.  The function $R_{\beta}(m,f)$ is a truncated Faber series of $f$.  It is completely determined by the sampling values of $f$ on the grid $G^d_\beta(m)$. Moreover,  $R_{\beta}(m,f)$ interpolates $f$ at the points of $G^d_\beta(m)$, i.e.,
\begin{equation}\nonumber
R_{\beta}(m,f)(\bx) 
= 
f(\bx), \quad \bx \in G^d_\beta(m).
\end{equation}
 As shown in what follows, with an appropriate choice of parameter $\beta$, the function $R_{\beta}(m,f)$ is suitable to approximately recovering the function $f$  in $\Uas$ from the sample values on the grid $G^d_\beta(m)$.

{We give a dimension-dependent error estimate of the approximation of a function} $f \in \Uas$ by the sampling operator $R_{\beta}(m,\cdot)$.

\begin{theorem} \label{thm:approx}
	Let $d \geq 2$,  $1 < \alpha \leq  2$,  $\beta>\alpha$, and $1\le p \le\infty$. Then for  every $f \in \Uas$  we have 
	\begin{equation}  \label{f-R}
	\|f -R_{\beta}(m,f) \|_{{ \mathring{W}_p^1}}
	\le K_1
	\frac{{ d^{2}}2^{-m(\alpha-1)}} { (p+1)^{\frac{d}{p}}2^{(\alpha + 1)d}\big(1-  2^{-\frac{\beta-\alpha}{\beta-1}}\big)^{d}}, 
	\end{equation}
	where $K_1=K_1(\alpha,\beta,p)=2(p+1)^{\frac{1}{p}} \max\big\{\frac{2\beta}{\beta-1},\frac{1}{2^{\alpha-1}-1}\big\}$.
\end{theorem}

A proof of this theorem is given  in Subsection \ref{proof of thm} in Appendix \ref{appendix}.

\begin{remark} \label{sparsity}
	{\rm
	For approximation of functions from $\Uas$,	we could take the sampling operator 
	$R_{\square}(m,f):= \sum_{|\bk|_\infty\le m} q_\bk(f)$ on 
	the  traditional standard grid 
$$
	G^d_\square(m):= \big\{ 2^{-\bk}\bs: |\bk|_\infty  = m, \  \bs \in Z^d(\bk+ \boned)\big\},
	$$ 
	and the sampling operator $R_{\triangle}(m,f):= \sum_{|\bk|_1 \le m} q_\bk(f)$ on the classical Smolyak grid 
	$$
	G^d_\triangle(m):= \big\{ 2^{-\bk}\bs:  |\bk|_1  = m, \  \bs \in Z^d(\bk+\boned)\big\}.
	$$
	 It is easy to verify that the error of approximation in the ${\mathring{W}_p^1}$ norm of $f \in \Uas$ by  $R_{\square}(m,f)$ or $R_{\triangle}(m,f)$ is the same as  by $R_{\beta}(m,f)$. On the other hand, the sparsity of   the grid $G^d_\beta(m)$ in the operator $R_{\beta}(m,f)$, is much higher than the sparsity of the grids $G^d_\square(m)$ and $G^d_\triangle(m)$, see the estimate of  $|G^d_\beta(m)|$ in Lemma \ref{card} in comparing with
	$|G^d_\square(m)| \approx 2^{dm}$ and $|G^d_\triangle(m)| \approx 2^{m+d} \binom{m+d-1}{d-1}$.  The sparse grids $G^d_\triangle(m)$ and related functions $R_{\triangle}(m,f)$ have been used in \cite{MoDu19, Suzu18} in construction of approximation by deep ReLU neural networks. In the next section, we will apply the sparse grids $G^d_\beta(m)$ and related functions $R_{\beta}(m,f)$ to approximation in $\mathring{W}^1_p$-norm  by deep ReLU neural networks of functions from $\Uas$.
}
\end{remark}
	\begin{remark} \label{rmk3.3}
		{\rm
	Some results similar to \eqref{f-R} were obtained in \cite[Theorem 3.8]{BuGr04} for the approximation in energy norm  of functions $f \in X^{p,2}$ with boundary zero value and mixed derivatives $\partial^\balpha f$, $|\balpha|_\infty\leq 2$,  bounded $L_p(\IId)$ for $p = 2, \infty$.  But the construction of sparse grids for approximation as well as high-dimensional technique of estimation of the approximation error cannot be applied to our case for the the mixed smoothness $1 < \alpha <2$. The authors of \cite{MoDu19} used the results and technique  of  	\cite{BuGr04}  in the $L_\infty(\IId)$-approximation by deep ReLU neural networks.
}
\end{remark}

\section{Approximation by deep ReLU neural networks}\label{sec:dnn}
{In this section, we will apply  the results on  sparse-grid sampling recovery in the previous section to the  approximation by deep  ReLU neural networks of functions from $\Uas$.}
For every $\varepsilon >0$ and  every $f \in \Uas$,  we will explicitly construct  a deep  ReLU neural network  $\Phi_f$ having an architecture ${\mathbb A}_\varepsilon$  independent of $f$, and the output $\mathcal{N}(\Phi_f,\cdot)$  which approximates $f$ in  the norm  of the isotropic Sobolev space { $\mathring{W}_p^1$} with accuracy $\varepsilon$, and give dimension-dependent upper bounds for the  size and the depth of $\Phi_f$. We also give some lower bounds for the  size of  deep ReLU neural network necessary for this approximation {in the case when $p=\infty$}. Up to logarithmic term our result for the  required size is optimal.

\subsection{Upper evaluation}
For upper evaluation of size of $\Phi_f$, our strategy is to use {the truncated Faber series}  $R_\beta(m,f)$ in Theorem \ref{thm:approx} as an  intermediate approximation, and then construct a deep  ReLU neural network  $\Phi_f$   for approximating  this sum by the output $\mathcal{N}(\Phi_f,\cdot)$. Since $R_\beta(m,f)$ is a sum of tensor products of hat functions, first of all we will  process the approximation {of such tensor products by deep  ReLU neural networks. This can be done based on the following lemma \cite[Proposition 2.6]{OSZ19}} on approximating by deep ReLU neural networks the product of $d$ numbers.

\begin{lemma} \label{prop:multi}
	{Let  $d\geq 2$ and $0 <\delta < 1$. Then}  we can explicitly construct  a deep  ReLU neural network  $\Phi_P$ such that 
	$$
	\sup_{ \bx \in { \IId}} \Bigg|\prod_{i=1}^d x_i - \mathcal{N}(\Phi_P,\bx) \Bigg| \leq \delta, 
	$$
	and
	$$
	\esssup_{ \bx \in { \IId}}\sup_{j=1,\ldots,d}\Bigg|\frac{\partial}{\partial x_j}\prod_{i=1}^d x_i - \frac{\partial}{\partial x_j}\mathcal{N}(\Phi_P,\bx) \Bigg|\leq \delta,
	$$
	where $\frac{\partial}{\partial x_j}$ denotes a weak partial derivative. Furthermore, there exists a constant $C>0$ independent of $\delta\in (0,1)$ and $d\in \N$ such that
	$$
	W(\Phi_P) \leq C d\log (d\delta^{-1}) 
	\quad \text{and}\quad
	L(\Phi_P)  \leq C\log d\log(d\delta^{-1}) \,.
	$$
	Moreover, if $x_j=0$ for some $j\in [d]$, then $\mathcal{N}(\Phi_P,\bx)=0$. 
\end{lemma}

The last statement $\mathcal{N}(\Phi_P,\bx)=0$ ($d=2$) when $x_1\cdot x_2=0$ was proved in \cite[Proposition 3.1]{ScZe19} (see also \cite[Proposition C.2]{GKP20} and \cite[Proposition 4.1]{OPS19}). But this implies that the statement also holds for general $d$ since the network $\Phi_P$ is constructed as an binary tree of the network $\Phi_P$ when $d=2$.
Inspecting the proof of   this lemma (the proofs of \cite[Proposition 2.6]{OSZ19} and  \cite[Proposition 3.3]{ScZe19}), 
we also find that when $\bx \in \IId$ it holds $N_w(\Phi_P)\leq 12 d$. Lemma \ref{prop:multi} yields the following lemma on approximating  tensor products of $d$ hat functions by deep ReLU neural networks.

\begin{lemma} \label{lem:product}
	 Let  $d\geq 2$ and $0 <\delta < 1$. Then  for the $d$-variate hat functions $\varphi_{\bk,\bs}$, $\bk \in \NNd_0$, $\bs \in Z^d(\bk)$,   defined as in \eqref{hat-function}, we can explicitly construct a deep ReLU neural network  $\Phi_{\bk,\bs}$ such that  $\mathcal{N}(\Phi_{\bk,\bs},\cdot)$  approximates $\varphi_{\bk,\bs}$ with accuracy $\delta$, and 
	$$
	N_w(\Phi_{\bk,\bs})\leq Cd,\quad	W(\Phi_{\bk,\bs}) \leq C d\log(d\delta^{-1}),\quad \text{and}\quad L(\Phi_{\bk,\bs} ) \leq C\log d\log(d\delta^{-1})\,.
	$$
	Moreover, $\supp(\mathcal{N}(\Phi_{\bk,\bs},\cdot))\subset \supp (\varphi_{\bk,\bs})$ and
	$$
	\esssup_{\bx \in \IId}\bigg|\frac{\partial}{\partial x_j}\varphi_{\bk,\bs} (\bx ) - \frac{\partial}{\partial x_j}\mathcal{N}(\Phi_{\bk,\bs},\bx) \bigg|\leq  2^{k_j+1}\delta \,.
	$$
\end{lemma}
\begin{proof}
	Indeed, we write
	$$
	y_i:=	\varphi_{k_i,s_i}(x_i) =\sigma\big(1-\sigma\big(2^{k_i+1}x_i - 2s_i-1\big) - \sigma\big(2s_i+1-2^{k_i+1}x_i\big)\big).
	$$
	Let $\Phi_P$ be the deep  ReLU neural network 	in Lemma \ref{prop:multi} and $\by=(y_1,\ldots,y_d)$ be the inputs of $\Phi_P$. Then we obtain a deep  ReLU neural network  denoted by $\Phi_{\bk,\bs}$.
	We have 
	$$\sup_{ \bx \in \IId} \bigg|\prod_{i=1}^d \varphi_{k_i,s_i}(x_i) - \mathcal{N}(\Phi_{\bk,\bs},\bx) \bigg|= \sup_{ \by \in \IId} \bigg|\prod_{i=1}^d y_i - \mathcal{N}(\Phi_P,\by) \bigg| \leq \delta  $$
	and
	$$	\esssup_{\bx \in \IId}\bigg|\frac{\partial}{\partial x_j}\varphi_{\bk,\bs} (\bx ) - \frac{\partial}{\partial x_j}\mathcal{N}(\Phi_{\bk,\bs},\bx) \bigg| = \esssup_{\by \in \IId}\Bigg|\Bigg(\frac{\partial}{\partial y_j}\prod_{i=1}^d y_i - \frac{\partial}{\partial y_j}\mathcal{N}(\Phi_P,\by)\Bigg) \frac{\dd y_j}{\dd x_j} \Bigg| \leq  2^{k_j+1}\delta \,.
	$$
	Moreover, we have 
	$$L(\Phi_{\bk,\bs})=L(\Phi_P)+2\quad \text{and}\quad W(\Phi_{\bk,\bs}) \leq W(\Phi_P)+7d.$$
	From Lemma \ref{prop:multi} we obtain the desired result.
	\hfill
\end{proof}

We are now ready to formulate and prove the main result.

\begin{theorem}\label{thm:dnn-upper} Let  $d \ge 2$, $1 < \alpha \le 2$, $\beta \geq \alpha$ and $1 \le p \le \infty$. Let 
	$$
	\varepsilon_0=\min \Bigg\{ 1,\frac{d}{2^{\alpha d}(1-2^{1-\alpha})},\frac{K_1{ d^{2}}} {(p+1)^{d/p}2^{(\alpha+1) d}\big(1-2^{-\frac{\beta-\alpha}{\beta-1}}\big)^d } \Bigg\},
	$$	
	where  $K_1$ is  the constant given in Theorem \ref{thm:approx}.  
	
	Then for every $\varepsilon\in (0,\varepsilon_0)$ we can explicitly construct a deep neural network architecture ${\mathbb A}_\varepsilon$ with the following property.
	For every $f \in \Uas$,  we can explicitly construct a deep   ReLU neural network  $\Phi_f$ having the architecture ${\mathbb A}_\varepsilon$ such that
	\begin{equation} \label{<epsilon}
	\|f- \mathcal{N}(\Phi_f,\cdot) \|_{{\mathring{W}_p^1}} \leq \varepsilon,
	\end{equation}
	and there hold the estimates
	\[
	L({\mathbb A}_\varepsilon)  \leq K_2  \log d\log(\varepsilon^{-1})\,
	\quad\text{and}\quad
	W({\mathbb A}_\varepsilon) \leq K_3 B^{-d}(\varepsilon^{-1})^{\frac{1}{\alpha-1}}\log(\varepsilon^{-1}),
	\]
		where $K_2=K_2(\alpha)$ and $K_3=K_3(\alpha,\beta,p)$ are positive constants, and
	\begin{equation} \label{eq:B2}
B=B(d,\alpha,\beta,p):= \big(1-2^{-\frac{1}{\beta-1}}\big)\Bigg(\frac { (p+1)^{\frac{1}{p}}2^{(\alpha+1) }\big(1-2^{-\frac{\beta-\alpha}{\beta-1}}\big)}{{ d^{\frac{2\alpha }{d}}}}\Bigg)^{\frac{1}{\alpha-1}}.
	\end{equation}

	Moreover, if  $\alpha$ and $p$ satisfy
	\begin{equation}\label{eq:condi-alpha-p}
	2^{\frac{1}{\alpha}}-(p+1)^{-\frac{1}{p\alpha}}/2>1
	\end{equation}  and
	\begin{equation}\label{eq:cond-beta}
	\frac{ \alpha+\log\big(1-(p+1)^{-\frac{1}{p\alpha}}2^{-1-\frac{1}{\alpha}}\big) }{1+\log\big(1-(p+1)^{-\frac{1}{p\alpha}}2^{-1-\frac{1}{\alpha}}\big)} <\beta<  1 -\frac{1}{\log\big(1-(p+1)^{-\frac{1}{p\alpha}}2^{-1-\frac{1}{\alpha}}\big)}
	\end{equation}
	then there exist constants $d_0(\alpha,\beta,p)\in \N$  and $B_0(\alpha,\beta,p) > 1$ such that $B \ge B_0(\alpha,\beta,p) > 1$ for all $d\geq d_0(\alpha,\beta,p)$.
\end{theorem}

\begin{proof}  We prove the theorem for the case $1\leq p < \infty$. The case $p = \infty$ can be carried out similarly with a slight modification. Let $f \in \Uas$. 
	For $\varepsilon\in (0,\varepsilon_0)$ we take 
	\begin{equation} \label{mm}
	m=\Bigg\lceil \frac{1}{\alpha -1} \log\Bigg(\frac{2K_1{ d^{2}}\varepsilon^{-1}} { (p+1)^{d/p}2^{(\alpha+1) d}\big(1-2^{-\frac{\beta-\alpha}{\beta-1}}\big)^d }\Bigg)\Bigg\rceil.
	\end{equation}
	Let $(\bk,\bs)\in D^d_\beta(m)$ and $\Phi_{\bk,\bs}$ be the  deep ReLU neural network  obtained in Lemma \ref{lem:product}. {Its output} $\mathcal{N}(\Phi_{\bk,\bs},\cdot)$ approximates $\varphi_{\bk,\bs}$ with accuracy $\delta=\delta(\varepsilon)\in (0,1)$ which is chosen later.
	Let $\Phi_f$ be the  deep ReLU neural network  obtained by parallelization as in Lemma \ref{lem:paralell} with the output
	$$\mathcal{N}(\Phi_f,\bx) = \sum_{\bk \in \Delta^d_\beta(m)}   \sum_{\bs\in Z^d(\bk)}  \lambda_{\bk,\bs}(f) \mathcal{N}(\Phi_{\bk,\bs},\bx),\quad \bx\in \IId.$$
	Then we can write
	\begin{equation}  \label{triangle-ineq}
	\|f-\mathcal{N}(\Phi_f,\cdot) \|_{{ \mathring{W}_p^1}} \leq \| f- R_{\beta}(m,f)\|_{{\mathring{W}_p^1}} + \|R_{\beta}(m,f)-\mathcal{N}(\Phi_f,\cdot)\|_{{\mathring{W}_p^1}},
	\end{equation}	
	where $R_{\beta}(m,f)$ is the operator given in \eqref{Rbm}.
	With the choice of $m$  {as} in \eqref{mm},	we get from Theorem~\ref{thm:approx}
	\begin{equation}  \label{|f-R|<}
	\begin{aligned} 
	\|f -R_{\beta}(m,f) \|_{{ \mathring{W}_p^1}}
	&\le K_1\frac{{ d^{2}}2^{-m(\alpha-1)}} { (p+1)^{\frac{d}{p}}2^{(\alpha + 1)d}\big(1-  2^{-\frac{\beta-\alpha}{\beta-1}}\big)^{d}} \le \varepsilon/2.
	\end{aligned}
	\end{equation}
	Let us estimate the norm $\|R_{\beta}(m,f) -\mathcal{N}(\Phi_f,\cdot)\|_{{ \mathring{W}_p^1}}$. Since $\supp\big(\mathcal{N}(\Phi_{\bk,\bs},\cdot)\big)\subset \supp(\varphi_{\bk,\bs})$  (see Lemma \ref{lem:product}), we have
	\begin{align*}
	&	\|R_{\beta}(m,f) -\mathcal{N}(\Phi_f,\cdot)\|_{{\mathring{W}_p^1}} 
	\leq 
	\sum_{\bk \in \Delta^d_\beta(m)} \Bigg\|\sum_{\bs\in Z^d(\bk)} \lambda_{\bk,\bs}(f) \big( \varphi_{\bk,\bs}- \mathcal{N}(\Phi_{\bk,\bs},\cdot)\big)\Bigg\|_{{\mathring{W}_p^1}}
	\\
	& =
	\sum_{\bk \in \Delta^d_\beta(m)}  \Bigg(\sum_{j=1}^d \int_{\IId}\Bigg|\sum_{\bs\in Z^d(\bk)} \lambda_{\bk,\bs}(f)\bigg(\frac{\partial}{\partial x_j}\varphi_{\bk,\bs} (\bx) - \frac{\partial}{\partial x_j}\mathcal{N}(\Phi_{\bk,\bs},\bx) \bigg)\Bigg|^p\dd \bx\Bigg)^{1/p}
	\\
	& =
	\sum_{\bk \in \Delta^d_\beta(m)}  \Bigg(\sum_{j=1}^d\sum_{\bs\in Z^d(\bk)} |\lambda_{\bk,\bs}(f)|^p\int_{\IId}\bigg|\frac{\partial}{\partial x_j}\varphi_{\bk,\bs} (\bx) - \frac{\partial}{\partial x_j}\mathcal{N}(\Phi_{\bk,\bs},\bx)\bigg|^p\dd \bx\Bigg)^{1/p}.
	\end{align*}
	Using Lemma \ref{lem:product} and estimate \eqref{eq:estimate} we get
	\beqq
	\begin{split}
		\|R_{\beta}(m,f) -\mathcal{N}(\Phi_f,\cdot)\|_{{\mathring{W}_p^1}}
		& \leq
		\sum_{\bk \in \Delta^d_\beta(m)} \sup_{\bs\in Z^d(\bk)}|\lambda_{\bk,\bs}(f)| \Bigg(\sum_{j=1}^d\sum_{\bs\in Z^d(\bk)} 2^{-|\bk|_1}2^{p(k_j+1)}\Bigg)^{1/p}  \delta
		\\
		&
		\leq 2^{-(\alpha+1)d}\sum_{\bk \in  \Delta^d_\beta(m) }2^{-|\bk|_1\alpha}\Bigg(\sum_{j=1}^d2^{p(k_j+1)}\Bigg)^{1/p}\delta
		\\
		& \leq 	2^{-(\alpha+1)d}	   \sum_{\ell=0}^\infty2^{-\ell \alpha} \sum_{|\bk|_1=\ell}\Bigg(\sum_{j=1}^d2^{p(k_j+1)}\Bigg)^{1/p} \delta.
	\end{split}
	\eeqq
	Now Lemma \ref{lem:p-induction} leads to
	\beqq
	\begin{split}
		\|R_{\beta}(m,f) -\mathcal{N}(\Phi_f,\cdot)\|_{{\mathring{W}_p^1}}
		& \leq d 2^d 2^{-(\alpha+1)d}  \delta \sum_{\ell=0}^\infty2^{-\ell (\alpha-1)}
		\leq 
		\frac{d2^{-\alpha d} }{1-2^{1-\alpha}} \delta
		.
	\end{split}
	\eeqq
	Define
	$
	\delta = \delta (\varepsilon): =
	\frac{1-2^{1-\alpha}}{d2^{-\alpha d}} \frac{\varepsilon}{2}.
	$
	Since $\varepsilon < \frac{2d}{2^{\alpha d}(1-2^{1-\alpha})}$ we get $\delta<1$. This choice of $\delta$ gives
	\begin{equation*}
	\begin{split}
	\|R_{\beta}(m,f) -\mathcal{N}(\Phi_f,\cdot)\|_{{\mathring{W}_p^1}}
	\le \varepsilon/2
	\end{split}
	\end{equation*}
	which together with \eqref{triangle-ineq} and 	\eqref{|f-R|<} proves \eqref{<epsilon}.
	
	We now prove the bounds for the  depth and the  size of $\Phi_f$. From  Lemmata \ref{lem:paralell} and \ref{lem:product} we have
	\begin{equation} \label{eq:L(phif)}
	\begin{split}
	L(\Phi_f)& =  \max_{(\bk,\bs)\in D_\beta^d(m)} L(\Phi_{\bk,\bs})
	\leq C\log d\log(d\delta^{-1})
	\\
	&\leq C \log d\log\bigg(\frac{2d^2 \varepsilon^{-1}}{2^{\alpha d}(1-2^{1-\alpha})}\bigg) \leq K_2\log d \log(\varepsilon^{-1})
	\end{split}
	\end{equation}
	for some positive constant $K_2=K_2(\alpha)$,
	and
	\begin{equation}\label{eq:W(phif)}
	\begin{split}
	W(\Phi_f) & \leq 
	\sum_{(\bk,\bs)\in D_\beta^d(m)} W(\Phi_{\bk,\bs})
	+ 
	\sum_{(\bk,\bs): L(\Phi_{\bk,\bs}) <L(\Phi_f)} \big(L(\Phi_f)-L(\Phi_{\bk,\bs})+2\big)
	\\
	&
	\leq 
	\sum_{(\bk,\bs)\in D_\beta^d(m)} W(\Phi_{\bk,\bs})
	+ 
	\sum_{(\bk,\bs)\in D_\beta^d(m)} \bigg( \max_{(\bk,\bs)\in D_\beta^d(m)} W(\Phi_{\bk,\bs})+2\bigg)
	\\
	&
	\leq  2|D_\beta^d(m)| \max_{(\bk,\bs)\in D_\beta^d(m)}\big( W(\Phi_{\bk,\bs} ) +  1\big)
	\leq \frac{2\beta}{\beta-1}\frac{d2^m}{\big(1-2^{-\frac{1}{\beta-1}}\big)^{d}} \big( Cd\log(d\delta^{-1})+1 \big),
	\end{split}
	\end{equation}
where in the last estimate we have used Lemma \ref{card}. From the choice of $m$ we derive 
	\begin{equation*} 
	\begin{split}
	d^22^m\big(1-2^{-\frac{1}{\beta-1}}\big)^{-d} 
	& \leq  2 d^2  \big(1-2^{-\frac{1}{\beta-1}}\big)^{-d} \Bigg(\frac{2K_1{ d^{2}}\varepsilon^{-1}} { (p+1)^{\frac{d}{p}}2^{(\alpha+1) d}\big(1-2^{-\frac{\beta-\alpha}{\beta-1}}\big)^d  }\Bigg)^{\frac{1}{\alpha-1}}
	\\
	& \leq 
	2 (2K_1)^{\frac{1}{\alpha-1}}\big(1-2^{-\frac{1}{\beta-1}}\big)^{-d} \Bigg(\frac{{ d^{\frac{2\alpha}{d}}}} { (p+1)^{\frac{1}{p}}2^{(\alpha+1) }\big(1-2^{-\frac{\beta-\alpha}{\beta-1}}\big)   }\Bigg)^{\frac{d}{\alpha-1}}(\varepsilon^{-1})^{\frac{1}{\alpha-1}}.
	\end{split}
	\end{equation*}
	Inserting this into \eqref{eq:W(phif)} we find
	\begin{equation*}
	W(\Phi_f) \leq K_3  B^{-d}(\varepsilon^{-1})^{\frac{1}{\alpha-1}}(\log\varepsilon^{-1})
	\end{equation*}
	with $B$ given in \eqref{eq:B2} and some positive constant $K_3$ depending on $\alpha$, $\beta$, and $p$. 
	
	To complete the proof of the first statement of the theorem it is sufficient to notice that $\Phi_f$ has the architecture ${\mathbb A}_\varepsilon$ (independent of $f$) which is defined as the minimal architecture of the deep ReLU neural network $\Phi$ obtained by  parallelization as in Lemma \ref{lem:paralell} with the output
	$$\mathcal{N}(\Phi,\bx) = \sum_{\bk \in \Delta^d_\beta(m)}   \sum_{\bs\in Z^d(\bk)}  \mathcal{N}(\Phi_{\bk,\bs},\bx),\quad \bx\in \IId.$$
	
To prove the second one, we show that under the conditions \eqref{eq:condi-alpha-p} and \eqref{eq:cond-beta} it holds
	\[
	\big(1-2^{-\frac{1}{\beta-1}}\big)^{(\alpha-1)}(p+1)^{\frac{1}{p}}2^{\alpha+1}\big(1-2^{-\frac{\beta-\alpha}{\beta-1}}\big) > 1.
	\]
	Indeed, the last inequality is equivalent to
	\begin{equation}\label{eq:beta}
	\frac{1}{(p+1)^{\frac{1}{p}}2^{\alpha+1}}<\big(1-2^{-\frac{1}{\beta-1}}\big)^{(\alpha-1)}\big(1-2^{-\frac{\beta-\alpha}{\beta-1}}\big).
	\end{equation}
	If $\beta-\alpha \geq 1$ then $\big(1-2^{-\frac{1}{\beta-1}}\big)\leq \big(1-2^{-\frac{\beta-\alpha}{\beta-1}}\big)$. Hence, the last inequality is fulfilled if 
	\[
	\frac{1}{(p+1)^{\frac{1}{p}}2^{\alpha+1}}<\big(1-2^{-\frac{1}{\beta
			-1}}\big)^\alpha \qquad \Longleftrightarrow \qquad \beta <1 -\frac{1}{\log\big(1-(p+1)^{-\frac{1}{p\alpha}}2^{-1-\frac{1}{\alpha}}\big)}.
	\]
	If $\beta-\alpha<1$ then $\big(1-2^{-\frac{1}{\beta-1}}\big)> \big(1-2^{-\frac{\beta-\alpha}{\beta-1}}\big)$. 
	Hence the inequality \eqref{eq:beta} is fulfilled if
	\[
	\frac{1}{(p+1)^{\frac{1}{p}}2^{\alpha+1}} <\big(1-2^{-\frac{\beta-\alpha}{\beta
			-1}}\big)^\alpha \qquad \Longleftrightarrow \qquad \beta > \frac{ \alpha+\log\big(1-(p+1)^{-\frac{1}{p\alpha}}2^{-1-\frac{1}{\alpha}}\big) }{1+\log\big(1-(p+1)^{-\frac{1}{p\alpha}}2^{-1-\frac{1}{\alpha}}\big)}.
	\]
The equivalence is due to $1-(p+1)^{-\frac{1}{p\alpha}}2^{-1-\frac{1}{\alpha}}>2^{-1}$. 	Assigning 
	$$ 1 -\frac{1}{\log\big(1-(p+1)^{-\frac{1}{p\alpha}}2^{-1-\frac{1}{\alpha}}\big)} > \frac{ \alpha+\log\big(1-(p+1)^{-\frac{1}{p\alpha}}2^{-1-\frac{1}{\alpha}}\big) }{1+\log\big(1-(p+1)^{-\frac{1}{p\alpha}}2^{-1-\frac{1}{\alpha}}\big)} $$
	we find $ 2^{\frac{1}{\alpha}}-(p+1)^{-\frac{1}{p\alpha}}/2>1$. 
	Since ${ d^{\frac{2\alpha }{d}}}$ tends to 1 when $d\to \infty$, there are $d_0(\alpha,\beta,p)\in \N$ and $B_0(\alpha,\beta,p) > 1$ such that $B \ge B_0(\alpha,\beta,p) > 1$ for all $d\geq d_0(\alpha,\beta,p)$. 
	\hfill
\end{proof}

For $f\in \Uas$ and $\varepsilon\in (0,\varepsilon_0)$, we observe that the width of the deep ReLU neural network  $\Phi_f$ constructed in Theorem \ref{thm:dnn-upper} may depend on $\varepsilon$. {To construct a deep neural network} with the same output that has a width independent of $\varepsilon$  we can concatenate the deep ReLU networks $\Phi_{\bk,\bs}$, $(\bk,\bs)\in D^d_\beta(m)$, with the help of special deep ReLU networks.

\begin{corollary}\label{cor:dnn-upper} Under the assumptions and notations of Theorem \ref{thm:dnn-upper},  for every $\varepsilon\in (0,\varepsilon_0)$ there exists a deep neural network architecture ${\mathbb A}_\varepsilon^*$ with the following property.
	For every $f \in \Uas$,  we can explicitly construct a deep   ReLU neural network  $\Phi_f^*$ having the architecture ${\mathbb A}_\varepsilon^*$, and positive constants $K_4$ and $K_5=K_5(\alpha,\beta,p)$ such that
	\begin{equation} \label{<epsilon1}
	\|f- \mathcal{N}(\Phi_f^*,\cdot) \|_{{ \mathring{W}_p^1}} \leq \varepsilon, 
	\end{equation}
	and there hold the estimates
	\begin{equation} \label{N_w&L}
	N_w({\mathbb A}_\varepsilon^*)  \leq K_4  d\,
	\quad\text{and}\quad
	L({\mathbb A}_\varepsilon^*) \leq 
	K_5 B^{-d}(\varepsilon^{-1})^{\frac{1}{\alpha-1}}\log(\varepsilon^{-1}).
	\end{equation}
\end{corollary}

\begin{proof}
	Consider the special deep ReLU neural network $\Phi_f'$ as in Figure \ref{fig:concatenation} ($d=2$). We number the set   $\{(\bk ,\bs )\in D_\beta^d(m)\}$ from $1$ to $J$, where $J=|D_\beta^d(m)|$. 	
	The $d$ source channels carry $\bx \in \IId$ forward so that it  is the input of all networks $\Phi_{\ell}$,  $\ell = 1,\ldots,J$. The $\big(\sum_{j=1}^\ell L(\Phi_j)\big)$th node in the collation channel stores the partial sum  
	$
	\sum_{j=1}^\ell \lambda_{j}(f) \mathcal{N}(\Phi_{j},\bx)
	$
	of the outputs of $\Phi_{\ell}$, $\ell=1,\ldots,J$. Hence,
	\begin{equation} \label{R_Phi=R_Phi'}
	\mathcal{N}(\Phi_f',\bx) = \mathcal{N}(\Phi_f,\bx)
	= \sum_{j=1}^J \lambda_{j}(f) \mathcal{N}(\Phi_{j},\bx),
	\end{equation}
	where $\Phi_f$ is the deep ReLU neural network in Theorem \ref{thm:dnn-upper}. From Lemma \ref{lem:product} we can find  an absolute positive constant $K_4$ so that
	$
	N_w(\Phi_f')\leq K_4d
	$
	and a positive constant $K_5=K_5(\alpha,\beta,p)$ such that
	$$
	L(\Phi'_f)\leq \sum_{j=1}^J L(\Phi_{j}) \leq  C |D^d_\beta(m)| \log d \log(d\delta^{-1}) 
	\leq 
	K_5 B^{-d}(\varepsilon^{-1})^{\frac{1}{\alpha-1}}\log(\varepsilon^{-1}),
	$$
	where the last inequality follows from \eqref{eq:L(phif)} and \eqref{eq:W(phif)}.  Hence, by Lemma \ref{lem:special} and \eqref{R_Phi=R_Phi'}, $\Phi_f'$  generates a deep ReLU neural network $\Phi_f^*$ such that $\mathcal{N}(\Phi_f^*,\bx) = \mathcal{N}(\Phi_f,\bx)$ and, consequently,  there hold \eqref{<epsilon1} and
	\[
	N_w(\Phi_f^*)  \leq K_4  d\,
	\quad\text{and}\quad
	L(\Phi_f^*) \leq 
	K_5  B^{-d}(\varepsilon^{-1})^{\frac{1}{\alpha-1}}\log(\varepsilon^{-1}).
	\] 
	The proof of existence of an architecture ${\mathbb A}_\varepsilon^*$ of $\Phi_f^*$ satisfying \eqref{N_w&L} is similar to the proof of existence of 
	${\mathbb A}_\varepsilon$ at the end of the proof of Theorem \ref{thm:dnn-upper}.
	\hfill
\end{proof}

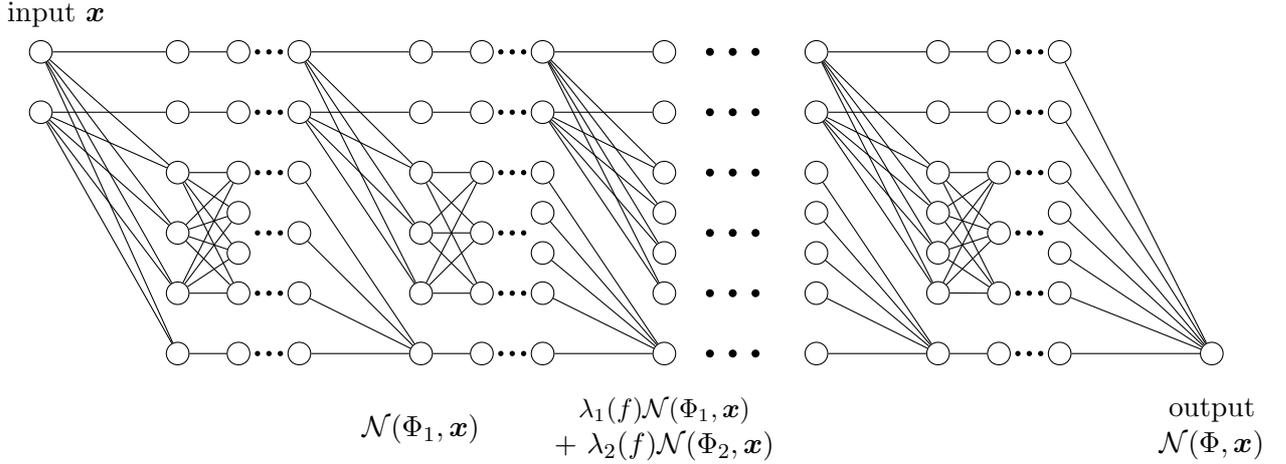
\begin{figure}
\begin{center}
\begin{tikzpicture}
\tikzstyle{place}=[circle, draw=black,scale=1, inner sep=3pt,minimum size=1pt, align=center]

\foreach \x in {1,...,2}
\draw node at (-0.2, -\x*0.8 +0.8) [place] (input_\x) { $$};
\node at (0, 0.5) {input $\bx$};

\foreach \x in {1,...,2}
\node at (1.6, -\x*0.8 +0.8) [place] (hiddenx1_\x){$$};

\foreach \x in {1,...,3}
\node at (1.6, -1.6 -\x*0.8+0.8) [place] (hiddenks1_\x){$$};

\node at (1.6, -4.0) [place] (hiddencol1){$$};

\foreach \x in {1,...,2}
\node at (2.4, -\x*0.8 +0.8) [place] (hiddenx15_\x){$$};

\foreach \x in {1,...,4}
\node at (2.4, -1.6 -\x*0.533+0.533) [place] (hiddenks15_\x){$$};

\node at (2.4, -4.0) [place] (hiddencol15){$$};

\foreach \x in {1,...,3}
\fill (2.5+\x*0.15, -0.0) circle (1pt);
\foreach \x in {1,...,3}
\fill (2.5+\x*0.15, -0.8) circle (1pt);
\foreach \x in {1,...,3}
\fill (2.5+\x*0.15, -1.6) circle (1pt);
\foreach \x in {1,...,3}
\fill (2.5+\x*0.15, -2.4) circle (1pt);
\foreach \x in {1,...,3}
\fill (2.5+\x*0.15, -3.2) circle (1pt);
\foreach \x in {1,...,3}
\fill (2.5+\x*0.15, -4.0) circle (1pt);

\foreach \x in {1,...,2}
\node at (3.2,  -\x*0.8 +0.8) [place] (hiddenx2_\x){$$};

\foreach \x in {1,...,3}
\node at (3.2, -1.6 -\x*0.8+0.8) [place] (hiddenks2_\x){$$};

\node at (3.2, -4.0) [place] (hiddencol2){$$};

\foreach \x in {1,...,2}
\node at (4.8, -\x*0.8 +0.8) [place] (hiddenx3_\x){$$};
\foreach \x in {1,...,3}
\node at (4.8, -1.6 -\x*0.8+0.8) [place] (hiddenks3_\x){$$};
\node at (4.8, -4.0) [place] (hiddencol3){$$};
\node at (4.8, -5.0)  {$\mathcal{N}(\Phi_{1},\bx)$};

\foreach \x in {1,...,2}
\node at (5.6,  -\x*0.8 +0.8) [place] (hiddenx35_\x){$$};
\foreach \x in {1,...,3}
\node at (5.6, -1.6 -\x*0.8+0.8) [place] (hiddenks35_\x){$$};
\node at (5.6, -4.0) [place] (hiddencol35){$$};

\foreach \x in {1,...,3}
\fill (5.7+\x*0.15, -0.0) circle (1pt);
\foreach \x in {1,...,3}
\fill (5.7+\x*0.15, -0.8) circle (1pt);
\foreach \x in {1,...,3}
\fill (5.7+\x*0.15, -1.6) circle (1pt);
\foreach \x in {1,...,3}
\fill (5.7+\x*0.15, -2.4) circle (1pt);
\foreach \x in {1,...,3}
\fill (5.7+\x*0.15, -3.2) circle (1pt);
\foreach \x in {1,...,3}
\fill (5.7+\x*0.15, -4.0) circle (1pt);

\foreach \x in {1,...,2}
\node at (6.4,  -\x*0.8 +0.8) [place] (hiddenx4_\x){$$};
\foreach \x in {1,...,4}
\node at (6.4, -1.6 -\x*0.533+0.533) [place] (hiddenks4_\x){$$};
\node at (6.4, -4.0) [place] (hiddencol4){$$};

\foreach \x in {1,...,2}
\node at (8.0,  -\x*0.8 +0.8) [place] (hiddenx5_\x){$$};
\foreach \x in {1,...,4}
\node at (8.0, -1.6 -\x*0.533+0.533) [place] (hiddenks5_\x){$$};
\node at (8.0, -4.0) [place] (hiddencol5){$$};
\node at (8.0, -5.0) [align=center] {$\small \lambda_1(f)\mathcal{N}(\Phi_{1},\bx)$ \\  + $ \lambda_2(f)\mathcal{N}(\Phi_{2},\bx)$};

\foreach \x in {1,...,3}
\fill (8.3+\x*0.3, -0.0) circle (1.5pt);
\foreach \x in {1,...,3}
\fill (8.3+\x*0.3, -0.8) circle (1.5pt);
\foreach \x in {1,...,3}
\fill (8.3+\x*0.3, -1.6) circle (1.5pt);
\foreach \x in {1,...,3}
\fill (8.3+\x*0.3, -2.4) circle (1.5pt);
\foreach \x in {1,...,3}
\fill (8.3+\x*0.3, -3.2) circle (1.5pt);
\foreach \x in {1,...,3}
\fill (8.3+\x*0.3, -4.0) circle (1.5pt);

\foreach \x in {1,...,2}
\node at (10.0, -\x*0.8 +0.8) [place] (hiddenx6_\x){$$};
\foreach \x in {1,...,4}
\node at (10.0, -1.6 -\x*0.533+0.533) [place] (hiddenks6_\x){$$};
\node at (10.0, -4.0) [place] (hiddencol6){$$};

\foreach \x in {1,...,2}
\node at (11.6, -\x*0.8 +0.8) [place] (hiddenx7_\x){$$};
\foreach \x in {1,...,4}
\node at (11.6, -1.6 -\x*0.532+0.532) [place] (hiddenks7_\x){$$};
\node at (11.6, -4.0) [place] (hiddencol7){$$};

\foreach \x in {1,...,2}
\node at (12.4, -\x*0.8 +0.8) [place] (hiddenx75_\x){$$};
\foreach \x in {1,...,3}
\node at (12.4, -1.6 -\x*0.8+0.8) [place] (hiddenks75_\x){$$};
\node at (12.4, -4.0) [place] (hiddencol75){$$};
\foreach \x in {1,...,3}
\fill (12.5+\x*0.15, -0.0) circle (1pt);
\foreach \x in {1,...,3}
\fill (12.5+\x*0.15, -0.8) circle (1pt);
\foreach \x in {1,...,3}
\fill (12.5+\x*0.15, -1.6) circle (1pt);
\foreach \x in {1,...,3}
\fill (12.5+\x*0.15, -2.4) circle (1pt);
\foreach \x in {1,...,3}
\fill (12.5+\x*0.15, -3.2) circle (1pt);
\foreach \x in {1,...,3}
\fill (12.5+\x*0.15, -4.0) circle (1pt);

\foreach \x in {1,...,2}
\node at (13.2, -\x*0.8 +0.8) [place] (hiddenx8_\x){$$};
\foreach \x in {1,...,4}
\node at (13.2, -1.6 -\x*0.533+0.533) [place] (hiddenks8_\x){$$};
\node at (13.2, -4.0) [place] (hiddencol8){$$};

\node at (15.2, -4.0)  [place] (output){$$};
\node at (15.2, -5.0) [align=center] {output
\\
$\mathcal{N}(\Phi,\bx)$
};

\foreach \i in {1,...,2}
\foreach \j in {1,...,3}
\draw [-] (input_\i) to (hiddenks1_\j);
\foreach \i in {1,...,2}
\draw [-] (input_\i) to (hiddencol1);
\foreach \i in {1,...,2}
\draw [-] (input_\i) to (hiddenx1_\i);


\foreach \i in {1,...,2}
\draw [-] (hiddenx1_\i) to (hiddenx15_\i);

\foreach \i in {1,...,3}
\foreach \j in {1,...,4}
\draw [-] (hiddenks1_\i) to (hiddenks15_\j);
\draw [-] (hiddencol1) to (hiddencol15);



\foreach \i in {1,...,2}
\draw [-] (hiddenx2_\i) to (hiddenx3_\i);
\foreach \i in {1,...,2}
\foreach \j in {1,...,3}
\draw [-] (hiddenx2_\i) to (hiddenks3_\j);

\foreach \j in {1,...,3}
\draw [-] (hiddenks2_\j) to (hiddencol3);
\draw [-] (hiddencol2) to (hiddencol3);


\foreach \i in {1,...,2}
\draw [-] (hiddenx3_\i) to (hiddenx35_\i);
\foreach \i in {1,...,3}
\foreach \j in {1,...,3}
\draw [-] (hiddenks3_\i) to (hiddenks35_\j);

\draw [-] (hiddencol3) to (hiddencol35);



\foreach \i in {1,...,2}
\draw [-] (hiddenx4_\i) to (hiddenx5_\i);
\foreach \i in {1,...,2}
\foreach \j in {1,...,4}
\draw [-] (hiddenx4_\i) to (hiddenks5_\j);

\foreach \j in {1,...,4}
\draw [-] (hiddenks4_\j) to (hiddencol5);
\draw [-] (hiddencol4) to (hiddencol5);



\foreach \i in {1,...,2}
\draw [-] (hiddenx6_\i) to (hiddenx7_\i);
\foreach \i in {1,...,2}
\foreach \j in {1,...,4}
\draw [-] (hiddenx6_\i) to (hiddenks7_\j);

\foreach \j in {1,...,4}
\draw [-] (hiddenks6_\j) to (hiddencol7);
\draw [-] (hiddencol6) to (hiddencol7);


\foreach \i in {1,...,2}
\draw [-] (hiddenx7_\i) to (hiddenx75_\i);
\foreach \i in {1,...,4}
\foreach \j in {1,...,3}
\draw [-] (hiddenks7_\i) to (hiddenks75_\j);

\draw [-] (hiddencol7) to (hiddencol75);


\foreach \i in {1,...,2}
\draw [-] (hiddenx8_\i) to (output);
\foreach \j in {1,...,4}
\draw [-] (hiddenks8_\j) to (output);

\draw [-] (hiddencol8) to (output);
\end{tikzpicture}
\caption{The graph  of the deep neural network $\Phi_f'$ as a concatenation of the neural networks $\Phi_{\bk,\bs}$, $(\bk,\bs)\in D_\beta^d(m)$, ($d=2$)}
\label{fig:concatenation}
\end{center}
\end{figure}

\begin{remark}
	
	{\rm
	 When $1 \le p \le 2$, the  condition \eqref{eq:condi-alpha-p} is satisfied for all $\alpha\in (1,2]$, and, therefore, we can always find  $\beta = \beta (\alpha)>\alpha$, $d_0(\alpha)\in \NN$, and $B_0(\alpha) > 1$  such that $B \ge B_0(\alpha) > 1$ for all $d\geq d_0(\alpha)$. 
}

\end{remark}

\begin{remark} \label{rmk4.3}
	{\rm	
 In \cite{Suzu18} the  author studied the approximation by deep ReLU neural networks of functions in the Besov space $B^\alpha_{p,\theta}(\IId)$ of mixed smoothness. The approximation error is measured in the norm of the space $L_\infty(\IId)$. If  $\alpha>\max(0,1/p -1/q)$ and  deep ReLU networks have depth $\mathcal{O}(\log N)$, width  $\mathcal{O}(N)$ and size $\mathcal{O}(N\log N)$, then the approximation error  is estimated by $C(\alpha, p, \theta, q, d)N^{-\alpha}\log^{\alpha(d-1)}N$.  When $\alpha\in (1,2)$,  the space $B^\alpha_{\infty,\infty}(\IId)$ coincides with $H_\infty^\alpha(\IId)$.  The approximation problem considered in the present paper,  is  completely different from that in \cite{Suzu18} since the error of approximation  is measured in the norm of the space $\mathring{W}_p^1$. This leads to  that the term $\log(\varepsilon^{-1})$ in Theorem \ref{thm:dnn-upper} does not depend on dimension $d$, compared to  the term $\log^{\alpha(d-1)}N$ in \cite{Suzu18} which is increasing exponentially  in $d$ when $d$ going to $\infty$. Moreover, { the constant} $C(\alpha, p, \theta, q, d)$ in \cite{Suzu18} is unexplicit in dimension $d$.}

\end{remark}

\begin{remark} \label{rmk4.4}
	{\rm			
			As  commented in Introduction and in Section \ref{sec:sampling}, in \cite{MoDu19} the authors constructed a deep ReLU  neural network for approximation  with accuracy $\varepsilon$ of a function with homogeneous boundary condition  in Sobolev space $ X^{p,2} \subset W^2_p(\IId) $ ($p=2,\infty$) of mixed smoothness $2$  (recall that $X^{\infty,2}=\mathring{H}_\infty^2(\IId)$). The approximation error is measured in the norm of the space $L_\infty(\IId)$. Its depth and size are evaluated as  $C(p,d) (\log(\varepsilon^{-1}) \log d)$ and $C'(p,d)(\varepsilon^{-1/2} \log^{\frac{3}{2}(d-1) + 1}(\varepsilon^{-1})(d - 1))$, respectively.   The term $\log^{\frac{3}{2}(d-1) + 1}(\varepsilon^{-1})$  increases faster than the exponent $B^d$ for any $B> 1$ when  $d$ becoming very large.  The constants $C(p,d)$and $C'(p,d)$ were not specified explicitly in dimension $d$.  
			 {The approximation problem investigated} in the present paper is also different from that in \cite{MoDu19},  since as mentioned above, the error of approximation  is measured in $\mathring{W}_p^1$. 
			Although  the representation by Faber series  was used in  both \cite{MoDu19} and our paper, as mentioned in Remark \ref{sparsity}, methods of sparse-grid approximation and techniques of evaluation of the approximation error are different, because at least the norms measuring the error are different. 		
		}		
\end{remark}

\begin{remark}\label{remark:PDE}
{\rm
  We now give an application of our results in this section to  numerical approximation of  solutions to elliptic PDEs. Consider a modeled diffusion elliptic equation  with homogeneous boundary condition:
\begin{equation} \label{ellip}
- {\rm div} (a(\bx)\nabla u(\bx))
\ = \
f(\bx) \quad \text{in} \quad \IId,
\quad u|_{\partial \IId} \ = \ 0, 
\end{equation}
 where the function $f$ and diffusion coefficient $a$ have sufficient regularity.
 Denote by $V:= \mathring{W}^1_2(\IId)$ the energy space. If $a$ satisfies the ellipticity assumption
\begin{equation} \nonumber
0<a_{\min} \leq a(\bx) \leq a_{\max}<\infty, \ \forall \bx \in \IId,
\end{equation}
by the well-known Lax-Milgram lemma, there exists a unique solution $u \in V$  in weak form which satisfies the variational equation
\begin{equation} \nonumber
\int_{\IId} a(\bx)\nabla u(\bx) \cdot \nabla v(\bx) \, \text{d} \bx
\ = \
\int_{\IId}  f(\bx)  v(\bx) \, \text{d} \bx  \quad \forall v \in V.
\end{equation}
In numerical implementation for solving the equation \eqref{ellip}, the error of an approximation  is measured in  the norm of the energy  space $V$(see, e.g., \cite{Cia78}).
Assume for the modeled case that $a$ and $f$ have H\"older-Nikol'skii mixed smoothness $2$, i.e., $a,f \in H^2_\infty(\IId)$. Then, the solution $u$ has at least mixed derivatives $\partial^{\balpha}u$ with $\balpha\in \NN_0^d$, $|\balpha|_\infty \leq 2$,  belonging to $L_2(\IId)$ \cite{GK09}, and therefore, by embedding for function spaces of mixed smoothness, see \cite[Theorem 2.4.1]{ST87B}, $u$ belongs to  $\mathring{H}^{3/2}_\infty(\IId)$.
For simplicity we assume that $ u \in \mathring{U}^{3/2}_\infty$. According to Theorem \ref{thm:dnn-upper}, for any $\varepsilon > 0$ sufficient small one can explicitly construct a deep ReLU neural networks $\Phi_u$ having  the output $\mathcal{N}(\Phi_u,\cdot)$ that approximates $u$ in  the  norm of the energy space $V$ with accuracy $\varepsilon$ such that it  holds the dimension-dependent estimates for the  size
\[
W(\Phi_u ) \leq C_1 B^{-d} \varepsilon^{-2} \log(\varepsilon^{-1}),
\]
with $B > 1$, and
the   depth  
\[
L(\Phi_u ) \leq  C_2 \log d \log(\varepsilon^{-1}).
\] 
Moreover, the continuous piece-wise linear function $\mathcal{N}(\Phi_u,\cdot)$  can be designed also as an output $\mathcal{N}(\Phi_u^*,\cdot)$ of    another   ``very" deep ReLU neural network $\Phi_u^*$ with the  $\varepsilon$-independent width $N_w(\Phi_u^*) \le C_3 d$  and
the  dimension-dependent depth  
\[
L(\Phi_u^*) \leq   C_4 B^{-d}\varepsilon^{-2}\log(\varepsilon^{-1}).
\]
	}
\end{remark}

\begin{remark} \label{rmk4.6}
	
	{\rm
		 As commented in Introduction and shown in Remark \ref{remark:PDE}, the H\"older-Zygmund function class  $\Uas$ and the norm of the isotropic Sobolev space $\mathring{W}_p^1$ naturally arise from some high-dimensional problems of  approximation and numerical methods  of PDEs. 	 On the other hand, due to the homogeneous boundary condition  of functions  from the class $\Uas$,  $d$-dimensional $\mathring{W}_p^1$-norm and $L_p(\IId)$-norm of any $f \in \Uas$ is  decreasing very fast when $d$ going to infinity. Indeed, for $1 < \alpha \le 2$, one can derive from \eqref{L&Wnorms} ($m=-1$) the inequality		
			\begin{equation*}  
	\sup_{f \in \Uas}	\|f\|_{\mathring{W}_p^1}
			\le  
			C(\alpha,p) M_1^{-d}
			\end{equation*}
	for  any number $M_1$ satisfying the condition  $1<M_1 < (p+1)^{1/p}2^{\alpha}$. 	A similar estimate 
	$$
	\sup_{f \in \Uas}	\|f\|_{L_p(\IId)} \le  M_0^{-d}
	$$  
	with some $M_0 >1$ holds true for the $L_p(\IId)$-norm, see \cite[Theorem 2.1]{DuTh20}.
	
There is an analogous picture for the problem of approximation in the $L_\infty(\IId)$-norm by deep ReLU neural networks considered in the paper \cite{MoDu19}, see Remark \ref{rmk4.4}. In this paper, the functions $f \in X^{p,2}$ to be approximated  also satisfy the homogeneous boundary condition. If 
$U^{p,2}$ is the unit ball of the space $ X^{p,2}$, then one can show that
	\begin{equation*}  
\sup_{f \in U^{p,2}}	\|f\|_{L_\infty(\IId)}
\le  
C(p) M_p^{-d}
\end{equation*}
for  some  number $M_p > 1$.  	
}	

\end{remark}

\subsection{Lower evaluation}

In the case when the  approximation error is measured in the norm of the space $\mathring{W}^1_{\infty}$, we are able to give dimension-dependent  lower  bounds   for the size  of deep ReLU networks whose outputs approximate  functions from $\Uas$  with a given accuracy. More precisely, we have the following results.

\begin{theorem}\label{thm:lower-dnn}
	Let $d\geq 2$ and $1 < \alpha \le 2$.  Let $\varepsilon\in(0,24^{-d})$ and $\mathbb{A}$ be a neural network architecture such that for any $f\in \Uas$, there is a deep  ReLU neural network $\Phi_f$ having  the architecture  $\mathbb{A}$ and
	\begin{equation*}\label{lowerbound01}
	\| f-\mathcal{N}(\Phi_f,\cdot) \|_{{\mathring{W}_\infty^1}} \leq \varepsilon.
\end{equation*}	
	Then there is a positive constant  $K_6=K_6(\alpha)$ such that
	$$
	W(\mathbb{A})\geq K_6 24^{-\frac{d}{2(\alpha-1)}}\varepsilon^{-\frac{1}{2(\alpha-1)}}.
	$$
	If assume in addition that 
	$$
	L(\mathbb{A})\leq C(\log\varepsilon^{-1})^\lambda
	$$
	for some constant $C>0$ and $\lambda\geq 0$, then there exists a constant  $K_7=K_7(\alpha)>0$ such that
		\begin{equation}\label{lowerbound02}
W(\mathbb{A})\geq K_7 24^{-\frac{d}{\alpha-1}}\varepsilon^{-\frac{1}{\alpha-1}} (\log\varepsilon^{-1})^{-{ \lambda}-1}.	
	\end{equation}	
\end{theorem}

 The results of this theorem are novel in the sense that the lower bounds are explicit in dimension $d$. Some non-dimension-dependent lower bounds have been obtained in \cite{Ya17a,GKP20}	for isotropic Sobolev spaces.  By using a recent result on VC-dimension  bounds for piecewise
	linear neural networks in \cite{BHLM19}  the lower bound in \eqref{lowerbound02}  in the case $d=1$ is also improved comparing with those of \cite{Ya17a,GKP20}.
In order to prove the above theorem we develop some techniques in \cite{Ya17a, GKP20} which are relied on upper bounds for VC-dimension of ReLU networks with Boolean outputs \cite{AnBa09B}.
We start with recalling a definition of VC-dimension.
\begin{definition}\label{def:vcdim}
	Let $H$ be a set of functions $h:X\to \{0,1\}$ for some set $X$. Then the VC-dimension of $H$, denoted by $\VCdim(H)$, is defined as the supremum of all number $m$ that there exist $x^1,\ldots,x^m\in X$ such that for any sequence $(y_j)_{j=1}^m\in \{0,1\}^m$ there is a function $h\in H$ with $h(x^j)=y_j$ for $j=1,\ldots,m$.
\end{definition}
The relation between VC-dimension and  size and  depth of a neural network architecture is given in the following lemma, 
  (see \cite[Theorems 8.7]{AnBa09B} and Equation (2) in  \cite{BHLM19}.)

\begin{lemma}\label{lem:vcdim-nn}
	Let $\mathbb{A}$ be a  deep neural network architecture. Let $F$ be the class of functions-outputs  $f= \mathcal{N}(\Phi,\cdot)$ of all  deep ReLU neural networks $\Phi$  having the architecture $\mathbb{A}$. Let $a\in \RR$ and $H$ be the class of all functions $h_f: \IId\to \{0,1\}$, $f \in F$, defined by threshold:
	$$h_f(\bx):=
	\begin{cases}
	1 & \text{if  } f(\bx)>a
	\\
	0 &\text{if  } f(\bx)\leq a.
	\end{cases}
	$$
	
	Then 
	$$
	\VCdim(H)\leq CW(\mathbb{A})^2
	$$
	for some positive constant $C$. Moreover, there exists a $C'>0$ such that
	$$\VCdim(H)\leq C'  L(\mathbb{A}) W(\mathbb{A})\log(W(\mathbb{A})).$$
\end{lemma}
The following elementary property of  deep ReLU neural networks has been proven in \cite[Lemma D.1]{GKP20} which is based on the piece-wise linearity of ReLU activation function. 
\begin{lemma}\label{lem:affine}
	Let $\Phi$ be a deep ReLU neural network, $\bxi\in (0,1)^d$, and $\bnu\in \RR^d$. Then there exists an open set $G=G(\bxi,\bnu)\subset (0,1)^d$ and $\delta=\delta(\bxi,\bnu)>0$ such that $\bxi +\lambda \delta \bnu \in \overline{G}$ for $\lambda\in [0,1]$ and $\mathcal{N}(\Phi,\cdot)$ is affine on $G$. 
\end{lemma}
We are now in position to prove Theorem \ref{thm:lower-dnn}. 
 Following  an idea in \cite{Ya17a} we use the  bounds of the VC-dimension of deep ReLU neural networks in Lemma \ref{lem:vcdim-nn}. 
 Since we have to establish lower bounds explicit in dimension $d$, we cannot employ bump functions as in \cite{Ya17a}.   Instead we will use quadric B-splines in the proof.

\begin{proof} Given $\varepsilon\in(0,24^{-d})$ and $m=m(\varepsilon)\in \NN$ which will be chosen later. 
	Denote $L=L(\mathbb{A})$, $W=W(\mathbb{A})$ the  depth and the  size of  $\mathbb{A}$. We assume that
	$$\mathbb{A}=\big((\bW^1,\bb^1),\ldots,(\bW^L,\bb^L) \big),$$
	where	 $\bW^\ell$ is an $N_\ell\times N_{\ell-1}$ matrix, and $\bb^\ell \in \RR^{N_\ell}$ with  $N_0=d$, $N_L=1$, and $N_1,\ldots,N_{L-1}\in \NN$. 
	Let $\Phi$ be  a deep ReLU neural network of the architecture $\mathbb{A}$. For $\bx=(x_1,\ldots,x_d)\in \IId$ and $\delta>0$ we put $\bar{\bx} =\big(x_1+ \frac{2^{-m-2}}{3},x_2,\ldots, x_d\big)$ and define the deep ReLU neural network $\Phi^{\delta}$ by parallelization construction in Lemma \ref{lem:paralell} with output
	$$
	\mathcal{N}(\Phi^{\delta},\bx):=\frac{\mathcal{N}(\Phi,\bar{\bx})-\mathcal{N}(\Phi,\bar{\bx}-\delta \bee^1)}{\delta},\quad \bee^1=(1,0,\ldots,0)\in \RR^d.
	$$
	Then $\Phi^{\delta}$ is a deep ReLU neural network having the architecture 
	$$
	\tilde{\mathbb{A}} = \left( \left(\begin{bmatrix} 
	\bI_d  \\
	\bI_d \\
	\end{bmatrix}, \begin{bmatrix} 
	\bee^1 \\
	\bee^1 \\
	\end{bmatrix} \right), \left(\begin{bmatrix} 
	\bW^1&0 \\
	0 & \bW^1 \\
	\end{bmatrix}, \begin{bmatrix} 
	\bb^1 \\
	\bb^1 \\
	\end{bmatrix} \right),\ldots,\left(\begin{bmatrix} 
	\bW^{L-1}&0 \\
	0 & \bW^{L-1} \\
	\end{bmatrix}, \begin{bmatrix} 
	\bb^{L-1} \\
	\bb^{L-1} \\
	\end{bmatrix} \right) \left([\bW^L,\bW^L], \bb^L \right)
	\right),
	$$
	where $\bI_d$ is the identity matrix of size $d$.  
	It is clear that $\tilde{\mathbb{A}}$ has depth and  size
	\begin{equation}\label{eq:weigh-A}
	\tilde{L}:=L(\tilde{\mathbb{A}})=L+1,\qquad \text{and}\qquad \tilde{W}:=
	W(\tilde{\mathbb{A}})=2W+2d+1.
	\end{equation}
	Let $a=a(m,d)\in \R$ be a constant which will be clarified later. For a  deep ReLU neural network $\tilde{\Phi}$ having architecture $\tilde{\mathbb{A}}$, we define the function
	$$  h(\tilde{\Phi},\bx) =
	\begin{cases}
	1 & \text{if}\ \ \mathcal{N}(\tilde{\Phi},\bx) >a/2,
	\\
	0 & \text{if}\ \  \mathcal{N}(\tilde{\Phi},\bx) \leq a/2,
	\end{cases}
	$$
	and the set
	$$
	H:=H(\tilde{\mathbb{A}})=\big\{   h(\tilde{\Phi},\bx) : \tilde{\Phi}\ \text{is a  deep ReLU neural network having architecture } \tilde{\mathbb{A}}\big\}.
	$$
	In the following we will show that $\VCdim(H)\geq 2^m.$ To this end,  we define 
		$$
		\bx^j=\bigg( 2^{-m}\Big(j-\frac{3}{4}\Big), \frac{1}{2},\ldots,\frac{1}{2}\bigg)\in (0,1)^d,\quad j=1,\ldots,2^m. 
		$$
This implies that $\bar{\bx}^j\in (0,1)^d$ for $j=1,\ldots,2^m$.   For every $\by=(y_1,\ldots,y_{2^m})\in \{0,1\}^{2^m}$, we  will construct a function $f_\by \in \Uas$ such  that the function 
		$h(\Phi^{\delta^*}_\by,\cdot)$	belongs to $H$ and satisfies 	
		\begin{equation}\label{h(bx^j)=y_j }
		h(\Phi^{\delta^*}_\by,\bx^j)=y_j 
			\end{equation}
		with an appropriate choice of $\delta^*$, where $\Phi_\by$ is the deep ReLU network having the architecture $\mathbb{A}$ and satisfying 	
		\begin{equation}\label{ |f_by- N(Phi_by)|}
		\| f_\by-\mathcal{N}(\Phi_\by,\cdot) \|_{\mathring{W}_\infty^1} \leq \varepsilon.
	\end{equation}

 Let  $M_3$ be the   quadric B-spline with  
		knots at the points $0,1,2,3$, i.e.,
		$$
		M_3(x)= \frac{1}{2}
		\begin{cases}
		x^2 & \text{if  }  0\leq x<1
		\\
		-2x^2+6x-3 & \text{if  }  1\leq x<2
		\\
		(3-x)^2 & \text{if  }  2\leq x<3
		\end{cases}
		$$
		and $M_3(x)=0$ otherwise. Let $\psi(x):=M_3(3x)$.
		We define the univariate non-negative functions $\psi_{k,s}$ by
		\begin{equation*} 
		\psi_{k,s}(x):= \ \psi(2^kx-s+1), \ k \in \NN_0, \ s = 1,\ldots,2^k.
		\end{equation*} 
		One can also verify that
		\begin{equation*} 
		\operatorname{supp} (\psi_{k,s}) \ = \ I_{k,s} \ =: [2^{-k}(s-1), 2^{-k}s], \quad 
		\operatorname{int}I_{k,s} \cap \operatorname{int}I_{k,s'} \ = \ \varnothing, \ s \not= s'.
		\end{equation*} 
	Let $\by=(y_1,\ldots,y_{2^m})\in \{0,1\}^{2^m}$ be given. We define
\begin{align*}
f_\by(\bx)
& =18^{-d} 2^{-\alpha m}\bigg(\sum_{j=1}^{2^m}y_j \psi_{m,j}(x_1) \bigg) \prod_{\ell=2}^d \psi_{0,1}(x_\ell)
=18^{-d} 2^{-\alpha m}\bigg(\sum_{j\in J} \psi_{m,j}(x_1) \bigg) \prod_{\ell=2}^d \psi_{0,1}(x_\ell),
\end{align*}
where $J=\{j: y_j=1\}$. 
We prove that $f_\by \in \Uas$ 
		by showing 
		\begin{equation} \label{[Delta^{2,u}_h (f_m,x)]}
		\Delta^{2,u}_\bh (f_\by,\bx) \ \le \ 
		{\prod_{\ell \in u} |h_\ell|^\alpha}, \ \bx \in \IId, \ \bh \in [-1,1]^d, \ u\subset [d],
		\end{equation}
		following partly in \cite{DuTh20}. 
		Let us prove this inequality for $u=[d]$ and $\bh \in \IId$, the general case of $u$ can be proven in a similar way with a slight modification. We have 
		\begin{equation*}
		\Delta^{2,[d]}_\bh (f_\by,\bx) 
		\ =  18^{-d} 2^{-\alpha m} \Delta_{h_1}^2 \Bigg(\sum_{j \in J} \psi_{m,j}(x_1) \Bigg) \prod_{\ell=2}^d \Delta_{h_\ell}^2 \psi_{0,1}(x_\ell).
		\end{equation*}
		For every univariate function $f$ having locally absolutely continuous derivative, we have the following representation
		\[
		\Delta^2_h(f, x) \ = \  h^2\int_{\RR} f^{''}(x+t)[h^{-1}M_2(h^{-1}t)]\, \dd t,\quad h\in \RR,
		\]
		 where   $M_2$ is the hat function defined at the beginning of Subsection \ref{subsec-sparse-grid}, see,  e.g.,  \cite[page 45]{DeLo93B}.  	By using this formula we get
		\begin{equation*} 
		\begin{split}
		\Delta_{h_1}^2 \Bigg(\sum_{j \in J} \psi_{m,j}(x_1) \Bigg)  
		&=  
		h_1^2
		\int_{\RR} \Bigg(\sum_{j \in J} \psi_{m,j}(t) \Bigg)'' h_1^{-1} M_2\big((h_1^{-1}(t-x_1)\big)\, \dd t  
		\\
		&= \, 9
		h_1^2 2^{2m}	\int_{\RR} \Bigg(\sum_{j \in J} \Chi^0_{m,j}(t)\Bigg)  h_1^{-1} M_2\big((h_1^{-1}(t-x_1)\big)\, \dd t ,  
		\end{split}
		\end{equation*}
		where $\Chi^0_{m,j}(t)=  \Chi_{I_{m,j}^1}-2\Chi_{I_{m,j}^2}+ \Chi_{I_{m,j}^3}  $ and $ \Chi_{I_{m,j}^i} $ are the characteristic functions of  the intervals
		$$I_{m,j}^i:=  \bigg[2^{-m}(j-1) + \frac{2^{-m}(i-1)}{3},\ 2^{-m}(j-1)+ \frac{2^{-m}i}{3}\bigg],\qquad i=1,2,3.$$
		If $(2^mh_1)\leq 1$ we have from $|\Chi^0_{m,j}|\leq 2$
		\begin{equation} \label{delta-x1}
		\begin{split}
		2^{-\alpha m}\Bigg|\Delta_{h_1}^2 \Bigg(\sum_{j \in J} \psi_{m,j}(x_1) \Bigg)  \Bigg|
		&= \,
		9 h_1^\alpha (2^mh_1)^{(2-\alpha)}\Bigg| \int_{\RR} \Bigg(\sum_{j \in J} \Chi^0_{m,j}(t)\Bigg)    h_1^{-1} M_2\big((h_1^{-1}(t-x_1)\big)\, \dd t \Bigg| \\[1.5ex] 
		&\le \,
		18 h_1^\alpha \int_{\RR}  h_1^{-1} M_2\big((h_1^{-1}(t-x_1)\big)\, \dd t = 18h_1^\alpha,
		\end{split}
		\end{equation}
		where in the last equality we used $\int_{\RR} M_2(t)\dd t=1$. If $(2^mh_1)>1 $ we have by changing variable
		\begin{equation*} 
		\begin{split}
		2^{-\alpha m}\Bigg|\Delta_{h_1}^2 \Bigg(\sum_{j \in J} \psi_{m,j}(x_1) \Bigg)  \Bigg|
		&=  9
		h_1^2 2^{m(2-\alpha)}	\int_{\RR} \Bigg(\sum_{j \in J} \Chi^0_{m,j}(x_1+h_1t)\Bigg)   M_2(t)\, \dd t .
		\end{split}
		\end{equation*}
		Denote $K_{m,j}=\supp\big(\Chi_{m,j}(x_1+h_1\cdot)\big)= \big[\frac{2^{-m}(j-1)-x_1}{h_1}, \ \frac{2^{-m}j-x_1}{h_1}\big]$. If $ K_{m,j}\subset [0,1]$ or $ K_{m,j} \subset [1,2] $
		we have
		$$ \int_{\RR}  \Chi^0_{m,j}(x_1+h_1t)   M_2(t)\dd t=\int_{\RR} \Big( \Chi_{I_{m,j}^1}-2\Chi_{I_{m,j}^2}+ \Chi_{I_{m,j}^3}\Big) (x_1+h_1t)   M_2(t)\dd t=0$$ and there are at most three $j^i \in J$, $i=0,1,2$ such that 
		$i\in \text{int}\big(  K_{m,j^i}\big)$, $i=0,1,2$.
		It is not difficult to verify that
		$$\bigg|	\int_{\RR}   \Chi^0_{m,j^1}(x_1+h_1t)   M_2(t)\, \dd t \bigg| \leq \frac{3}{2} \Big( \frac{1}{3 \cdot 2^mh_1}\Big)^2 $$
		and
		$$\bigg|	\int_{\RR}   \Chi^0_{m,j^i}(x_1+h_1t)   M_2(t)\, \dd t \bigg| \leq \frac{1}{2} \Big( \frac{1}{3 \cdot 2^mh_1}\Big)^2 $$
		if $i=0,2$. 
		From this we obtain
		\begin{equation*} 
		\begin{split}
		2^{-\alpha m}\Bigg|\Delta_{h_1}^2 \Bigg(\sum_{j \in J} \psi_{m,j}(x_1) \Bigg)  \Bigg|
		&\leq  9
		h_1^2 2^{m(2-\alpha)}\sum_{i=0,1,2 }\Bigg|	\int_{\RR}   \Chi^0_{m,j^i}(x_1+h_1t)   M_2(t)\, \dd t \Bigg|
		\\
		&\leq  9
		h_1^2 2^{m(2-\alpha)}	\frac{5}{2}\Big( \frac{1}{3 \cdot 2^mh_1}\Big)^2
		= \frac{5}{2} h^{\alpha} \frac{1}{(2^mh_1)^{\alpha}} \leq \frac{5}{2} h_1^{\alpha}.
		\end{split}
		\end{equation*}
	Since $h_\ell\in [0,1]$, $\ell=2,\ldots,d$, similar to \eqref{delta-x1} we can show that 
		$$ \Bigg| \prod_{\ell=2}^d \Delta_{h_\ell}^2 \psi_{0,1}(x_\ell) \Bigg| \leq 18^{d-1} \prod_{\ell=2}^dh_\ell^\alpha.$$
		Consequently, the inequality \eqref{[Delta^{2,u}_h (f_m,x)]} is proven.
		This means that $f_\by \in \Uas$.	
			
Moreover, we have
	\begin{align*}
	\frac{\partial f_\by}{\partial x_1}(\bar{\bx}^j)
	=18^{-d} 2^{-\alpha m} y_j\psi_{m,j}'\Big( x^j_1 +\frac{2^{-m-2}}{3}\Big)\prod_{\ell=2}^d \psi( x^j_\ell)
	=y_j 18^{-d} 2^{-\alpha m} \psi\Big(\frac{1}{2}\Big)^{(d-1)}  2^{m} \psi'\Big(\frac{1}{3}\Big).
	\end{align*}
	From
	$
	\psi\big(\frac{1}{2}\big) = M_3\big(\frac{3}{2}\big) =\frac{3}{4}$ and $\psi'\big( \frac{1}{3}\big)=3 M'_3(1) =3
	$
	we get
	\begin{align*}
	\frac{\partial f_\by}{\partial x_1}(\bar{\bx}^j)
	&
	=4y_j 18^{-d} 2^{-(\alpha-1)m}  \Big(\frac{3}{4}\Big)^d  = 4y_j   2^{-(\alpha-1)m}   24^{-d}.
	\end{align*}
	Since $f_\by\in \Uas$, by the assumption, there exists a neural network $\Phi_\by$ having  the architecture $\mathbb {A}$ and satisfying \eqref{ |f_by- N(Phi_by)|}.
	
	 By Lemma \ref{lem:affine}, for $\bar{\bx}^j \in (0,1)^d$, there exists an open set  $G_j\subset (0,1)^d$ and $\delta_j>0$ such that $\bar{\bx}^j +\lambda \delta_j \bee^1 \in \bar{G}_j$ for $\lambda\in [0,1]$ and $\mathcal{N}(\Phi_\by,\cdot)$ is affine on $\bar{G}_j$. Let
	\begin{equation*}
	\delta^* = \min_{1\leq j\leq 2^m} \delta_j >0
	\end{equation*}
	and $B(\bar{\bx}^j,\theta_j)$ be the open ball centered at $\bar{\bx}^j$ with radius $\theta_j$. 
	Then we have 
	\begin{align*}
	\mathcal{N}(\Phi^{\delta^*}_\by,\bx^j)
	&=\frac{\mathcal{N}(\Phi_\by,\bar{\bx}^j)-\mathcal{N}(\Phi_\by,\bar{\bx}^j- \delta^*  \bee^1 )}{\delta^*}=\frac{\partial \mathcal{N}(\Phi_\by,{\bxi}^j)}{\partial x_1},
	\end{align*}
	for some $\bxi^j \in B(\bar{\bx}^j,\theta_j)\cap G_j$. In case $\bar{\bx}^j\in G_j$ we can choose $\bxi^j = \bar{\bx}^j$. 
	Define $m$ the largest positive integer such that
	$$
	\varepsilon \leq 2^{-(\alpha-1)m}   24^{-d}. 
	$$
	We then obtain for $y_j=1$,
	\begin{align*}
	\mathcal{N}(\Phi^{\delta^*}_\by,\bx^j)
	&= \frac{\partial  f_\by(\bar{\bx}^j)}{\partial x_1}+\frac{\partial \mathcal{N}(\Phi_\by,\bxi^j)}{\partial x_1} - \frac{\partial f_\by(\bar{\bx}^j)}{\partial x_1}
	\\
	& \geq \frac{\partial  f_\by(\bar{\bx}^j)}{\partial x_1} - \bigg|\frac{\partial \mathcal{N}(\Phi_\by,\bxi^j)}{\partial x_1} - \frac{\partial  f_\by(\bxi^j)}{\partial x_1}\bigg|-\bigg|\frac{\partial  f_\by( \bxi^j)}{\partial x_1} - \frac{\partial  f_\by(\bar{\bx}^j)}{\partial x_1}\bigg|.
	\end{align*}
	Since $ \frac{\partial  f_\by(\bx)}{\partial x_1} $ is a continuous function due to $f_\by\in \Uas$, $\alpha>1$. Then we can choose $\theta_j$ small enough such that 
	$$
	\bigg|\frac{\partial  f_\by( \bxi^j)}{\partial x_1} - \frac{\partial  f_\by(\bar{\bx}^j)}{\partial x_1}\bigg| \leq \frac{\varepsilon}{2}
	$$
	which implies
	\begin{align*}
	\mathcal{N}(\Phi^{\delta^*}_\by,\bx^j)
	&
	\geq  4\cdot    2^{-(\alpha-1)m}   24^{-d} - \frac{3}{2} \varepsilon 
	\geq \frac{5}{2}\cdot  2^{-(\alpha-1)m}   24^{-d}.
	\end{align*}
	For $y_j=0$,
	\begin{align*}
	\big|\mathcal{N}(\Phi^{\delta^*}_\by,\bx^j)\big|
	&= \bigg|\frac{\partial \mathcal{N}(\Phi_\by,\bxi^j)}{\partial x_1} - \frac{\partial  f_\by(\bxi^j)}{\partial x_1}+ \frac{\partial  f_\by(\bxi^j)}{\partial x_1} - \frac{\partial  f_\by(\bar{\bx}^j)}{\partial x_1}\bigg|
	\\
	& 
	\leq  \bigg|\frac{\partial \mathcal{N}(\Phi_\by,\bxi^j)}{\partial x_1} - \frac{\partial  f_\by(\bxi^j)}{\partial x_1}\bigg| +\bigg| \frac{\partial  f_\by(\bxi^j)}{\partial x_1} - \frac{\partial  f_\by(\bar{\bx}^j)}{\partial x_1}\bigg|\leq \frac{3}{2} \varepsilon 
	\leq   \frac{3}{2} \cdot 2^{-(\alpha-1)m}   24^{-d}.
	\end{align*}
	Putting $a=4\cdot 2^{-(\alpha-1)m}   24^{-d},
	$ we get
	$$  
	\begin{cases}
	\mathcal{N}(\Phi^{\delta^*}_\by,\bx^j) >a/2 & \text{if}\ \ y_j=1,
	\\
	\mathcal{N}(\Phi^{\delta^*}_\by,\bx^j)  \leq a/2 & \text{if}\ \ y_j=0.
	\end{cases}
	$$
	Then  the function
	$h(\Phi_\by^{\delta^*},\cdot)$
	belongs to $H$ and satisfies \eqref{h(bx^j)=y_j }.

	By definition of VC-dimension, see Definition \ref{def:vcdim}, we obtain $2^m \leq \VCdim(H)$. Moreover, from \eqref{eq:weigh-A} and Lemma \ref{lem:vcdim-nn} we have
	$$2^m \leq \VCdim(H)\leq C (2W+ 2d+1)^2.$$
	Since  $\mathbb{A}$ has input dimension $d$ and  depth $L\geq 2$,   we find that $2d+1\leq 2W$. From this and $2^{-(\alpha-1)(m+1)}24^{-d}\leq \varepsilon$ we get
	$$ 
	C4^2 W^2\geq 2^m \geq \frac{1}{2} 24^{-\frac{d}{\alpha-1}}\varepsilon^{-\frac{1}{\alpha-1}}
	$$
	or 
	$$W \geq 2^m \geq \frac{1}{4\sqrt{2C}}24^{-\frac{d}{2(\alpha-1)}}\varepsilon^{-\frac{1}{2(\alpha-1)}}$$
	which is the first statement. 
	
	Concerning second one, we have
	\begin{align*}
	\ \frac{1}{2} 24^{-\frac{d}{\alpha-1}}\varepsilon^{-\frac{1}{\alpha-1}} 
	&
	\leq C'  { \tilde{L}}\,   \tilde{W}\log\tilde{W} 
	= C' { (L+1)} (2W+2d+1)\log(2W+2d+1) 
	\\
	&
	\leq C''  (\log\varepsilon^{-1})^{{ \lambda}} W\log W 
	\end{align*}
	which implies 
	$$W\log W \geq  \frac{1}{2} 24^{-\frac{d}{\alpha-1}}\varepsilon^{-\frac{1}{\alpha-1}} (\log\varepsilon^{-1})^{-{ \lambda}} .$$
	Consider for ${ \kappa\leq 1}$, 
	$$ 
	W_\kappa= \kappa 24^{-\frac{d}{\alpha-1}}\varepsilon^{-\frac{1}{\alpha-1}} (\log\varepsilon^{-1})^{-{ \lambda}-1}.
	$$ 
	Then we have that
	$$\log W_\kappa= \log \kappa -\frac{d\log24}{\alpha-1}-({ \lambda}+1)\log(\log \varepsilon^{-1})+\frac{1}{\alpha-1}(\log\varepsilon^{-1})\leq \frac{1}{\alpha-1}(\log\varepsilon^{-1}).$$
	From this we obtain
	\begin{align*}
	W_\kappa\log W_\kappa 
	\leq \frac{\kappa}{\alpha-1} 24^{-\frac{d}{\alpha-1}}\varepsilon^{-\frac{1}{\alpha-1}} (\log\varepsilon^{-1})^{-{ \lambda}}
	\leq \frac{1}{2C''} 24^{-\frac{d}{\alpha-1}}\varepsilon^{-\frac{1}{\alpha-1}} (\log\varepsilon^{-1})^{-{ \lambda}} \leq W\log(W)
	\end{align*}
	if we choose   $\kappa \leq \min\{1, \frac{\alpha-1}{2C''}\}$. Consequently, we get
	$$\kappa 24^{-\frac{d}{\alpha-1}}\varepsilon^{-\frac{1}{\alpha-1}} (\log\varepsilon^{-1})^{-{ \lambda}-1} = W_\kappa\leq W$$ 
	which is the second statement. 
	\hfill
\end{proof}

\section{ Concluding remarks}\label{sec:conclusion}
  We  have explicitly constructed a deep ReLU neural network $\Phi_f$ having an output that approximates with an arbitrary  prescribed accuracy $\varepsilon$ in the norm of the isotropic Sobolev space {$\mathring{W}_p^1$  functions  $f \in \Uas$  having H\"older-Zygmund mixed smoothness $\alpha$ with  $1 <\alpha \le 2$}. For this approximation, we have established  a dimension-dependent estimate for the computation complexity characterized by   the size $W(\Phi_f)$ and the depth $L(\Phi_f)$ of  this deep ReLU neural network:     
 $$
 W(\Phi_f) \leq C_1  B^{-d}(\varepsilon^{-1})^{\frac{1}{\alpha-1}}\log(\varepsilon^{-1}) \quad {\rm and} \quad L(\Phi_f)  \leq  C_2  \log d\log(\varepsilon^{-1})
 $$
  with $B >1$.   

 This shows in particular, that the computation complexity is decreasing as fast as the exponent $B^{-d}$ when the dimension $d$ going to $\infty$. In the case when $p=\infty$, we  gave dimension-dependent  lower  bounds   for the size  of deep ReLU networks whose outputs approximate  functions from $\Uas$  with a given accuracy.  
The  analysis also indicated that the { representation of functions from the H\"older-Zygmund space of mixed smoothness $\Ha$ by}  tensor product Faber series  plays a fundamental role in construction  of deep ReLU neural networks for approximation of functions from  $\Ha$.  

In the present paper, our concerns are the non-adaptive  approximation by deep ReLU neural networks for which the architecture of  deep ReLU neural networks is the same for all functions.  In the recent paper \cite{DKT21}, we have investigated  a problem of  adaptive nonlinear approximation by deep ReLU neural networks of  multivariate functions having a mixed smoothness.

\appendix
\section{Appendix: Auxiliary results and proof of Theorem \ref{thm:approx}} \label{appendix}

\subsection{Auxiliary results}
\begin{lemma} \label{lem:qk-in-Wp}
	Let $d\in \NN$,  $1\leq \alpha \le 2$, and $1\leq p\le \infty$. Then for a function 
	$f \in \Uas$ and $\bk \in \NNd_0$ we have
	\begin{equation*} 
	\|q_\bk(f) \|_{{ \mathring{W}_p^1}}
	\le \ 
	\frac{2^{- \alpha|\bk|_1+1}}{(p+1)^{\frac{d-1}{p}}2^{(\alpha + 1)d}}
	\,    |2^{\bk}|_{p} .
	\end{equation*}  
\end{lemma}

\begin{proof} 
	Let us prove the lemma for the case $1\leq p < \infty$. The case $p=\infty$ can be proven similarly with a slight modification.	
	For  $\bk \in \NNd_0$, by disjoint supports of $\varphi_{\bk,\bs}$, $\bs\in Z^d(\bk)$, we have 
	\begin{equation*} 
	\begin{aligned}
	\|q_\bk(f) \|_{{ \mathring{W}_p^1}}^p
	\, &= \, 
	\sum_{i=1}^d \int_{\IId}\Bigg| \sum_{\bs\in Z^d(\bk)} \, \lambda_{\bk,\bs}(f) \frac{\partial}{\partial x_i}\varphi_{\bk,\bs} (\bx) \Bigg|^p \, \dd \bx  
	\\ 
	\ &\leq \ 
	\sup_{\bs\in Z^d(\bk)} \, |\lambda_{\bk,\bs}(f)|^p\sum_{i=1}^d   |Z^d(\bk)|\Bigg(\prod_{j \not = i}2 \int_{0}^{2^{-k_j-1}}|2^{k_j+1}x_j|^p \, \dd x_j
	\Bigg)\Bigg(2 \int_{0}^{2^{-k_i-1}}2^{pk_i+p}\dd x_i \Bigg)\\
	\ &\le \ 
	\bigg(\frac{2^{- \alpha|\bk|_1}}{2^{ (\alpha + 1)d} \, }  \bigg)^p \sum_{i=1}^d 
	2^{|\bk|_1} \Bigg(\prod_{j \not = i}\frac{2}{p+1}  2^{-k_j-1} 
	\Bigg) 2\cdot 2^{(k_i+1)(p-1)}
	\\
	& = \bigg(\frac{2^{- \alpha|\bk|_1}}{2^{ (\alpha + 1)d} \, }  \bigg)^p \frac{1}{(p+1)^{d-1}} \sum_{i=1}^d 
	2^{p(k_i+1)}
	.
	\end{aligned}
	\end{equation*}
	This proves the claim. 
	\hfill \end{proof}

The following lemma  gives an upper estimate of the cardinality of  $D^d_\beta(m)$ and $G^d_\beta(m)$ showing their  sparsity.

\begin{lemma}\label{card}  Let $d\in \NN$. We have for every $\beta> 1$ and $m\in \N$,
	$$
	\big|D^d_\beta(m)\big|\leq   \frac{\beta}{\beta-1}d\Big(1-2^{-\frac{1}{\beta-1}}\Big)^{-d}2^m,
	$$
	and hence,
	$$
	\big|G^d_\beta(m)\big|\leq   \frac{\beta}{\beta-1}d2^d\Big(1-2^{-\frac{1}{\beta-1}}\Big)^{-d}2^m.
	$$
\end{lemma}
\begin{proof}  Since $\big| G^d_{\beta}(m,f)\big| \le \big|D^d_\beta(m+d)\big|$, it is sufficient to prove the first estimate in the lemma. We have
	\begin{equation*} 
	\begin{split}
	|D^d_\beta(m)| =\sum_{\bk\in  \Delta^d_\beta(m)} 2^{|\bk|_1}
	= 
	2^m \sum_{j=0}^{m} 2^{-j}\sum_{|\bk|_1=m-j\atop |\bk|_\infty\geq m-\lfloor \beta j \rfloor} 1.	
	\end{split}
	\end{equation*}
	Note, that for $i\in  \NN_0$ and $\ell \in [d]$
	there are 
	$
	\binom{d-2+ \lfloor (\beta -1)j\rfloor -i}{d-2}
	$
	multi-indices $\bk\in \NNd_0$ satisfying $$k_\ell=m-\lfloor\beta j\rfloor+i,\qquad \text{and}\qquad \sum_{r\not =\ell}k_r=\lfloor (\beta -1)j\rfloor -i .$$
	From this we can estimate
	\begin{equation*}
	\sum_{|\bk|_1=m-j\atop |\bk|_\infty\geq m-\lfloor\beta j\rfloor} 1 
	\leq d \sum_{i=0}^{\lfloor(\beta -1)j\rfloor} \binom{d-2+ \lfloor (\beta -1)j\rfloor -i}{d-2}
	= d \binom{d-1+\lfloor (\beta -1)j\rfloor}{d-1}.
	\end{equation*}
	Hence, 
	\begin{equation*} 
	\begin{split}
	|D^d_\beta(m)| 
	&\leq d 2^m \sum_{j=0}^{m} 2^{-j} \binom{d-1+\lfloor (\beta -1)j\rfloor}{d-1}.
	\end{split}
	\end{equation*}
	Putting $\lfloor (\beta -1)j\rfloor=k$ we obtain $j\geq \frac{k}{\beta-1}$ and 
	$$\big|\{j\in \N_0: \lfloor (\beta -1)j\rfloor=k \}\big|<\frac{1}{\beta-1}+1$$
	which leads to
	\begin{equation*} 
	|D^d_\beta(m)| 
	\leq d 2^m \Big(\frac{1}{\beta-1}+1\Big) \sum_{j=0}^{\infty} 2^{-\frac{1}{\beta -1}j} \binom{d-1+j}{d-1}
	= d2^m \frac{\beta}{\beta-1}  \big(1-2^{-\frac{1}{\beta-1}}\big)^{-d}.
	\end{equation*}
	The last equality is due to  
	\begin{equation}\label{eq:series}
	\sum_{j=0}^\infty x^j\binom{k+j}{k}=(1-x)^{-k-1}
	\end{equation} for $k\in \N_0$ and $x\in (0,1)$ which is obtained by taking $k$th derivative both sides of $(1-x)^{-1}=\sum_{j=0}^\infty x^j$. The proof is completed.
	\hfill \end{proof}

\begin{lemma}\label{lem:p-induction}
	Let $d\in \NN$, $\ell\in \NN_0$, and $1\leq p\leq \infty$. Then it holds
	\begin{equation} \label{eq:induction}
	\sum_{\bk \in \NNd_0, |\bk|_1=\ell } |2^{\bk}|_{p} \leq d2^{\ell+d-1}. 
	\end{equation}
\end{lemma}
\begin{proof} By monotonicity  in $p$ of $\ell_p$-norms, it is enough to prove the lemma for $p=1$. We use induction argument with respect to $d$. It is obvious that the inequality holds for $d=1$ and $\ell\in \NN_0$ or $d\in \NN$ and $\ell=0$. Assume that 
	$$
	\sum_{\bk \in \NNd_0, |\bk|_1=j }    \sum_{i=1}^d 2^{k_i} \leq d 2^{j+d-1}
	$$
	for $j=0,\ldots,\ell$. We show that the inequality \eqref{eq:induction} holds for $d+1$ instead of $d$. Indeed, we have
	\begin{align*}
	\sum_{\bk \in \NN_0^{d+1}, |\bk|_1=\ell }    \sum_{i=1}^{d+1} 2^{k_i}
	& = \sum_{j=0}^\ell \sum_{k_1+\ldots+k_d=\ell-j} \bigg( 2^{j}+\sum_{i=1}^d 2^{k_i} \bigg)
	\\
	& \leq  \sum_{j=0}^\ell d2^{\ell-j+d-1} + \sum_{j=0}^\ell 2^{j}\binom{\ell-j+d-1}{d-1}
	\\
	&
	= d2^{d-1}(2^{\ell+1}-1) + 2^{\ell} \sum_{j=0}^\ell 2^{j-\ell} \binom{\ell-j+d-1}{d-1}.
	\end{align*}
	Using \eqref{eq:series} we finally obtain
	\begin{align*}
	\sum_{\bk \in \NN_0^{d+1}, |\bk|_1=\ell }    \sum_{i=1}^{d+1} 2^{k_i}
	& \leq d2^{d-1}(2^{\ell+1}-1) + 2^{\ell+d}
	\leq (d+1)2^{\ell+d}.
	\end{align*}
	The proof is completed.
	\hfill \end{proof}

\subsection{Proof of Theorem \ref{thm:approx}} \label{proof of thm}
\begin{proof}  Let us prove the theorem for the case $1\leq p < \infty$. The case $p = \infty$ can be proven similarly with a slight modification. {As noticed, the set} 
	$\Delta^d_\beta(m)$ with $\beta> 1$ is a subset of $\Delta^d(m)$. Hence, for  every $f \in \Uas$ we have 
	\begin{equation}  \label{eq:f-r}
	\begin{aligned}
	\|f -R_{\beta}(m,f) \|_{{ \mathring{W}_p^1}}
	&\le  \sum_{\bk\in \NNd_0 \backslash \Delta^d_\beta(m)} \|q_\bk(f) \|_{{ \mathring{W}_p^1}} 
	\\
	&= \sum_{\bk\in \NNd_0 \backslash \Delta^d(m)} \|q_\bk(f) \|_{{ \mathring{W}_p^1}}  + \sum_{\bk\in  \Delta^d(m) \backslash \Delta^d_\beta(m)} \|q_\bk(f) \|_{{\mathring{W}_p^1}}.
	\end{aligned}
	\end{equation}
	For the first sum in the right side, from Lemma \ref{lem:qk-in-Wp} we have 
	\begin{equation*}   
	\begin{aligned} 
	\sum_{\bk\in \NNd_0 \backslash \Delta^d(m)} \|q_\bk(f) \|_{{ \mathring{W}_p^1}} 
	&\le  \frac{1}{(p+1)^{\frac{d-1}{p}}2^{(\alpha + 1)d} } 
	\sum_{\bk \in \NNd_0, |\bk|_1>m} 2^{- \alpha|\bk|_1+1} |2^{\bk}|_{p}
	\\
	&= \frac{2}{(p+1)^{\frac{d-1}{p}}2^{(\alpha + 1)d} }  
	\sum_{  \ell =m+1}^\infty 2^{- \alpha \ell} \sum_{\bk \in \NNd_0, |\bk|_1=\ell }  |2^{\bk}|_{p}
	.
	\end{aligned}
	\end{equation*}
	In view of Lemma \ref{lem:p-induction} we get
	\begin{equation}  \label{L&Wnorms}
	\begin{aligned} 
	\sum_{\bk\in \NNd_0 \backslash \Delta^d(m)} \|q_\bk(f) \|_{{ \mathring{W}_p^1}}
	&\le  
	\frac{d}{(p+1)^{\frac{d-1}{p}}2^{\alpha d} }   
	\sum_{  \ell =m+1}^\infty 2^{- (\alpha-1) \ell}  
	= \frac{d2^{-(\alpha - 1)m}}{(p+1)^{\frac{d-1}{p}}2^{\alpha d}(2^{\alpha-1}-1)}.
	\end{aligned}
	\end{equation}
	We now consider the second sum in the right  side of \eqref{eq:f-r}. Denote {by} $j^*$ the maximum value of $j$ such that the set
	\begin{equation*}
	\big\{\bk\in \NNd_0:\ |\bk|_1=m-j,\ |\bk|_\infty <m-\lfloor\beta j\rfloor\big\}
	\end{equation*}
	is not empty. Following the argument in the proof of \cite[Theorem 3.10]{BuGr04}  and using  Lemma \ref{lem:qk-in-Wp} we get
	\begin{equation}  \label{eq:diff}
	\begin{aligned}
	\sum_{\bk \in  \Delta^d(m) \backslash\Delta^d_\beta(m)}\|q_\bk(f) \|_{{\mathring{W}_p^1}} 
	&
	\le  
	\sum_{j=0}^{j^*} \sum_{|\bk|_1=m -j\atop|\bk|_\infty <m -\lfloor\beta j\rfloor}\|q_\bk(f) \|_{{\mathring{W}_p^1}}
	\\
	&
	\leq \frac{2} {  (p+1)^{\frac{d-1}{p}}2^{(\alpha + 1)d} } \sum_{j=0}^{j^*} \sum_{|\bk|_1=m -j\atop|\bk|_\infty <m -\lfloor\beta j\rfloor} 
	2^{- \alpha|\bk|_1} |2^{\bk}|_{p}
	\\
	&= \frac{2^{-\alpha m+1}} { (p+1)^{\frac{d-1}{p}}2^{(\alpha + 1)d}} \sum_{j=0}^{j^*} 2^{\alpha j}\sum_{|\bk|_1=m -j\atop|\bk|_\infty <m -\lfloor\beta j\rfloor} 
	|2^{\bk}|_{p}.
	\end{aligned}
	\end{equation}
	{The last sum in the right side of \eqref{eq:diff}} can be estimated as 
	\begin{equation*}  
	\begin{aligned}
	\sum_{j=0}^{j^*} 2^{\alpha j}\sum_{|\bk|_1=m -j\atop|\bk|_\infty <m -\lfloor\beta j\rfloor} 
	|2^{\bk}|_{p}
	& \leq \sum_{j=0}^{j^*} 2^{\alpha j}\sum_{i=1}^{m-1-\lfloor\beta j\rfloor} { d} \binom{m+d-2-i-j}{d-2} 2^i
	\\
	& ={ d} 2^{m-1 }\sum_{j=0}^{j^*}  2^{-\lfloor(\beta-\alpha) j\rfloor}\sum_{\ell=0}^{m-2-\lfloor\beta j\rfloor}  \binom{d-1+\lfloor (\beta -1)j\rfloor+\ell}{d-2} 2^{-\ell},
	\end{aligned}
	\end{equation*}
	where in the equality we put $i=m-1-\lfloor\beta j\rfloor-\ell $. As in the proof of Lemma \ref{card}, we put $\lfloor (\beta -1)j\rfloor=k$ and get $\lfloor(\beta -\alpha) j\rfloor \geq \frac{k(\beta-\alpha)}{\beta-1}-1$. Since $\frac{\beta-\alpha}{\beta-1}<1$, we have 
	\begin{equation*}
	\begin{split}
	\sum_{j=0}^{j^*} 2^{\alpha j}\sum_{|\bk|_1=m -j\atop|\bk|_\infty <m -\lfloor\beta j\rfloor} 
	|2^{\bk}|_{p}
	&\leq \frac{d2^m\beta}{\beta-1} \sum_{k=0}^{\infty}  
	\sum_{\ell=0}^{\infty}  \binom{d-1+k+\ell}{d-2} 2^{-\ell}2^{-\frac{\beta-\alpha}{\beta-1}k}
	\\
	& \leq \frac{d2^m\beta}{\beta-1} \sum_{k=0}^{\infty}  
	\sum_{\ell=0}^{\infty}  \binom{d-1+k+\ell}{d-2}  2^{-\frac{\beta-\alpha}{\beta-1}(k+\ell)}
	\\
	& =\frac{d2^m\beta}{\beta-1}\sum_{j=0}^{\infty}  \binom{d-1+ j}{d-2}  2^{-\frac{\beta-\alpha}{\beta-1}j}(j+1)
	\\
	& =(d-1)\frac{d2^m\beta}{\beta-1} \sum_{j=0}^{\infty}  \binom{d-1 +j}{d-1}  2^{-\frac{\beta-\alpha}{\beta-1}j} .
	\end{split} 
	\end{equation*}
	From the assumption $\beta>\alpha>1$ and  \eqref{eq:series} we arrive at
	\[
	\sum_{j=0}^{j^*} 2^{\alpha j}\sum_{|\bk|_1=m -j\atop|\bk|_\infty <m -\lfloor\beta j\rfloor} 
	|2^{\bk}|_{p}\leq 
	\frac{d^2 2^m \beta}{\beta-1}\big(1-  2^{-\frac{\beta-\alpha}{\beta-1}}\big)^{-d}.
	\]
	Inserting this into \eqref{eq:diff} we  get
	\begin{equation*}  
	\begin{aligned}
	\sum_{\bk \in  \Delta^d(m) \backslash\Delta^d_\beta(m)}\|q_\bk(f) \|_{{\mathring{W}_p^1}}
	&\leq \frac{2\beta}{\beta-1}\frac{{ d^2}2^{-m(\alpha-1)}} { (p+1)^{\frac{d-1}{p}}2^{(\alpha + 1)d}\big(1-  2^{-\frac{\beta-\alpha}{\beta-1}}\big)^{d}}   .
	\end{aligned}
	\end{equation*}
	Since 
	$  2^{-d}\big(1-  2^{-\frac{\beta-\alpha}{\beta-1}}\big)^{-d}>1$
{for}  $\alpha>1$, we finally obtain the desired estimate.
	\hfill \end{proof}

\bigskip
\noindent
{\bf Acknowledgments.}  This work is funded by Vietnam National Foundation for Science and Technology Development (NAFOSTED) under  Grant No. 102.01-2020.03. A part of this work was done when  the authors were working at the Vietnam Institute for Advanced Study in Mathematics (VIASM). They would like to thank  the VIASM  for providing a fruitful research environment and working condition.

\bibliographystyle{abbrv}

\bibliography{AllBib}

\begin{thebibliography}{10}

\bibitem{AlNo20}
M.~Ali and A.~Nouy.
\newblock {Approximation of smoothness classes by deep ReLU networks}.
\newblock {\em arXiv:2007.15645}, 2020.

\bibitem{AnBa09B}
M.~Anthony and P.~Bartlett.
\newblock {\em {Neural Network Learning: Theoretical Foundations}}.
\newblock Cambridge University Press, Cambridge, 2009.

\bibitem{ABMM17}
R.~Arora, A.~Basu, P.~Mianjy, and A.~Mukherjee.
\newblock Understanding deep neural networks with rectified linear units.
\newblock {\em Electronic Colloquium on Computational Complexity}, Report No.
  98, 2017.

\bibitem{BHLM19}
P.~L. Bartlett, N.~Harvey, C.~Liaw, and A.~Mehrabian.
\newblock {Nearly-tight VC-dimension and pseudodimension bounds for
  piecewiselinear neural networks}.
\newblock {\em J. Mach. Learn. Res.}, 20:1--17, 2019.

\bibitem{Bel57B}
R.~Bellmann.
\newblock {\em {Dynamic Programming}}.
\newblock Princeton University Press, Princeton, 1957.

\bibitem{BuGr99}
H.-J. Bungartz and M.~Griebel.
\newblock {A note on the complexity of solving Pois-son’s equation for spaces
  of bounded mixed derivatives}.
\newblock {\em J. Complexity}, 15:167--199.

\bibitem{BuGr04}
H.-J. Bungartz and M.~Griebel.
\newblock { Sparse grids}.
\newblock {\em Acta Numer.}, 13:147--269, 2004.

\bibitem{BDSU16}
G.~Byrenheid, D.~{D\~ung}, W.~Sickel, and T.~Ullrich.
\newblock {Sampling on energy-norm based sparse grids for the optimal recovery
  of Sobolev type functions in $H_\gamma$}.
\newblock {\em J. Approx. Theory}, 207:207--231, 2016.

\bibitem{ChD16}
A.~Chernov and {D. D\~ung}.
\newblock {New explicit-in-dimension estimates for the cardinality of
  high-dimensional hyperbolic crosses and approximation of functions having
  mixed smoothness}.
\newblock {\em J. Complexity}, 32:92--121, 2016.

\bibitem{Cia78}
P.~Ciarlet.
\newblock {\em The Finite Element Method for Elliptic Problems}.
\newblock North Holland Publishing Company, 1978.

\bibitem{DU13}
{D. D\~ung and T. Ullrich}.
\newblock N-widths and $\varepsilon$-dimensions for high-dimensional
  approximations.
\newblock {\em Found. Comput. Math.}, 13:965--1003, 2013.

\bibitem{Dung11a}
D.~{D\~ung}.
\newblock {B-spline quasi-interpolant representations and sampling recovery of
  functions with mixed smoothness}.
\newblock {\em J. Complexity}, 27:541--567, 2011.

\bibitem{Dung11b}
D.~{D\~ung}.
\newblock {Optimal adaptive sampling recovery}.
\newblock {\em Adv. Comput. Math}, 34:1--41, 2011.

\bibitem{Dung16}
D.~{D\~ung}.
\newblock {Sampling and cubature on sparse grids based on a B-spline
  quasi-interpolation}.
\newblock {\em Found. Comp. Math.}, 16:1193--1240, 2016.

\bibitem{DKT21}
D.~{D\~ung}, V.~K. Nguyen, and M.~X. Thao.
\newblock {Computation complexity of deep ReLU neural networks in
  high-dimensional approximation}.
\newblock {\em arXiv.org/abs/2103.00815}, 2021.

\bibitem{DTU18B}
D.~{D\~ung}, V.~N. Temlyakov, and T.~Ullrich.
\newblock {\em {Hyperbolic Cross Approximation}}.
\newblock Advanced Courses in Mathematics - CRM Barcelona,
  Birkh\"auser/Springer, 2018.

\bibitem{DuTh20}
D.~{D\~ung} and M.~X. Thao.
\newblock {Dimension-dependent error estimates for sampling recovery on Smolyak
  grids based on B-spline quasi-interpolation}.
\newblock {\em J. Approx. Theory}, 250:185--205, 2020.

\bibitem{DDF.19}
I.~Daubechies, R.~DeVore, S.~Foucart, B.~Hanin, and G.~Petrova.
\newblock {Nonlinear approximation and (Deep) ReLU networks}.
\newblock {\em arXiv:1905.02199}, 2019.

\bibitem{DeLo93B}
R.~DeVore and G.~Lorentz.
\newblock {\em {Constructive Approximation}}.
\newblock Springer-Verlag, New York, 1993.

\bibitem{EWa18}
W.~E and Q.~Wang.
\newblock {Exponential convergence of the deep neural network approximation for
  analytic functions}.
\newblock {\em Sci. China Math.}, 61:1733--1740, 2018.

\bibitem{GGT01}
J.~Garcke, M.~Griebel, and M.~Thess.
\newblock Data mining with sparse grids.
\newblock {\em Computing}, 67:225--253, 2001.

\bibitem{GKNV19}
R.~Gribonval, G.~Kutyniok, M.~Nielsen, and F.~Voigtlaender.
\newblock Approximation spaces of deep neural networks.
\newblock {\em arXiv:1905.01208}, 2019.

\bibitem{GK09}
M.~Griebel and S.~Knapek.
\newblock Optimized general sparse grid approximation spaces for operator
  equations.
\newblock {\em Math. Comp.}, 78:2223--2257, 2009.

\bibitem{GPEB19}
P.~Grohs, D.~Perekrestenko, D.~Elbrachter, and H.~Bolcskei.
\newblock {Deep neural network approximation theory}.
\newblock {\em arXiv: 1901.02220}, 2019.

\bibitem{GKP20}
I.~G\"uhring, G.~Kutyniok, and P.~Petersen.
\newblock {Error bounds for approximations with deep ReLU neural networks in
  $W^{s,p}$ norms}.
\newblock {\em Anal. Appl. (Singap.)}, 18:803--859, 2020.

\bibitem{Heb49B}
D.~Hebb.
\newblock {\em {The Organization of Behavior: A Neuropsychological Theory}}.
\newblock Wiley, 1949.

\bibitem{KSH12}
A.~Krizhevsky, I.~Sutskever, and G.~E. Hinton.
\newblock {ImageNet classification with deep convolutional neural networks}.
\newblock {\em NeurIPS}, pages 1106--1114, 2012.

\bibitem{KSU15}
T.~K\"uhn, W.~Sickel, and T.~Ullrich.
\newblock {Approximation of mixed order Sobolev functions on the $d$-torus –
  Asymptotics, preasymptotics and $d$-dependence}.
\newblock {\em Constr. Approx.}, 42:353--398, 2015.

\bibitem{LBH15}
Y.~LeCun, Y.~Bengio, and G.~Hinton.
\newblock Deep learning.
\newblock {\em Nature}, 521:436--444, 2015.

\bibitem{MP43}
W.~S. McCulloch and W.~Pitts.
\newblock A logical calculus of the ideas immanent in nervous activity.
\newblock {\em Bull. Math. Biophys.}, 5:115--133, 1943.

\bibitem{Mha96}
H.~N. Mhaskar.
\newblock {Neural networks for optimal approximation of smooth and analytic
  functions}.
\newblock {\em Neural Comput.}, 8:164--177, 1996.

\bibitem{MoDu19}
H.~Montanelli and Q.~Du.
\newblock {New error bounds for deep ReLU networks using sparse grids}.
\newblock {\em SIAM J. Math. Data Sci.}, 1:78--92, 2019.

\bibitem{MPCB14}
G.~Mont\'ufar, R.~Pascanu, K.~Cho, and Y.~Bengio.
\newblock On the number of linear regions of deep neural networks.
\newblock {\em In Advances in neural information processing systems}, pages
  2924--2932, 2014.

\bibitem{Ni75B}
S.~M. Nikolskii.
\newblock {\em {Approximation of Functions of Several Variables and Embedding
  Theorems}}.
\newblock Springer, Berlin, 1975.

\bibitem{NoWo08}
E.~Novak and H.~Wo{\'z}niakowski.
\newblock {\em {Tractability of Multivariate Problems, Volume I: Linear
  Information}}.
\newblock EMS Tracts in Mathematics, Vol. 6, Eur. Math. Soc. Publ. House,
  Z\"urich, 2008.

\bibitem{NoWo10}
E.~Novak and H.~Wo{\'z}niakowski.
\newblock {\em {Tractability of Multivariate Problems, Volume II: Standard
  Information for Functionals}}.
\newblock EMS Tracts in Mathematics, Vol. 12, Eur. Math. Soc. Publ. House,
  Z\"urich, 2010.

\bibitem{OPS19}
J.~A.~A. Opschoor, P.~C. Petersen, and C.~Schwab.
\newblock {Deep ReLU networks and high-order finite element methods}.
\newblock {\em Anal. Appl. (Singap.)}, 18:715--770, 2020.

\bibitem{OSZ19}
J.~A.~A. Opschoor, C.~Schwab, and J.~Zech.
\newblock { Exponential ReLU DNN expression of holomorphic maps in high
  dimension}.
\newblock {\em SAM, Research Report No. 2019-35}, 2019.

\bibitem{PeVo18}
P.~Petersen and F.~Voigtlaender.
\newblock {Optimal approximation of piecewise smooth functions using deep ReLU
  neural networks}.
\newblock {\em Neural Netw.}, 108:296--330, 2018.

\bibitem{Ros58}
F.~Rosenblatt.
\newblock {The perceptron: a probabilistic model for information storage and
  organization in the brain}.
\newblock {\em Psychol. Rev.}, 65:386--408, 1958.

\bibitem{ST87B}
H.~Schmeisser and H.~Triebel.
\newblock {\em {Topics in Fourier Analysis and Function Spaces}}.
\newblock Chichester; New York : Wiley, 1987.

\bibitem{ScZe19}
C.~Schwab and J.~Zech.
\newblock {Deep learning in high dimension: Neural network expression rates for
  generalized polynomial chaos expansions in UQ}.
\newblock {\em Anal. Appl. (Singap.)}, 17:19--55, 2019.

\bibitem{Suzu18}
T.~Suzuki.
\newblock {Adaptivity of deep ReLU network for learning in Besov and mixed
  smooth Besov spaces: optimal rate and curse of dimensionality}.
\newblock {\em International Conference on Learning Representations}, 2019.

\bibitem{Te15}
M.~Telgarsky.
\newblock Representation benefits of deep feedforward networks.
\newblock {\em arXiv:1509.08101}, 2015.

\bibitem{Te16}
M.~Telgarsky.
\newblock Benefits of depth in neural nets.
\newblock {\em In Proceedings of the JMLR: Workshop and Conference Proceedings,
  New York, NY, USA}, 49:1--23, 2016.

\bibitem{Tem18B}
V.~Temlyakov.
\newblock {\em {Multivariate Approximation}}.
\newblock Cambridge University Press, 2018.

\bibitem{Tem93B}
V.~N. Temlyakov.
\newblock {\em {Approximation of Periodic Functions}}.
\newblock Computational Mathematics and Analysis Series, Nova Science
  Publishers, Inc., Commack, NY., 1993.

\bibitem{Tri10B}
H.~Triebel.
\newblock {\em {Bases in Function Spaces, Sampling, Discrepancy, Numerical
  Integration}}.
\newblock European Math. Soc. Publishing House, Z\"urich, 2010.

\bibitem{Tri15B}
H.~Triebel.
\newblock {\em {Hybrid Function Spaces, Heat and Navier-Stokes Equations}}.
\newblock European Mathematical Society, 2015.

\bibitem{Vy06}
J.~Vybiral.
\newblock Function spaces with dominating mixed smoothness.
\newblock {\em Diss. Math.}, 436:1--73, 2006.

\bibitem{WSC.16}
Y.~Wu, M.~Schuster, Z.~Chen, Q.~V. Le, and M.~Norouzi.
\newblock {Google’s neural machine translation system: Bridging the gap
  between human and machine translation}.
\newblock {\em arXiv: 1609.08144, 2016}.

\bibitem{Ya17a}
D.~Yarotsky.
\newblock {Error bounds for approximations with deep ReLU networks}.
\newblock {\em Neural Netw.}, 94:103--114, 2017.

\bibitem{Ya17b}
D.~Yarotsky.
\newblock {Quantified advantage of discontinuous weight selection in
  approximations with deep neural networks}.
\newblock {\em arXiv: 1705.01365}, 2017.

\bibitem{Yser10}
H.~Yserentant.
\newblock {\em {Regularity and Approximability of Electronic Wave Functions}}.
\newblock Lecture Notes in Mathematics, Springer, 2010.

\end{thebibliography}

\end{document}